\documentclass[10pt,eqright]{amsart}
\usepackage{amsmath}
\usepackage{latexsym}
\usepackage{amssymb}
\usepackage{graphics}

\input epsf%
\newtheorem{theorem}{Theorem}[section]

\newtheorem{remark}{Remark}[section]

\newtheorem{Definition}{Definition}[section]
\def\qed{{\hfill{\vrule height4pt width3pt depth2pt}}}

   \font\caa=cmsy12\def\Cal#1{\hbox{\caa #1}}
\long\def\mcomment#1{}

\def\ad#1{\begin{aligned}#1\end{aligned}}  
\def\a#1{\begin{align*}#1\end{align*}} \def\an#1{\begin{align}#1\end{align}}
\def\e#1{\begin{equation}#1\end{equation}} \def\d{\operatorname{div}}
\def\p#1{\begin{pmatrix}#1\end{pmatrix}} 
  \numberwithin{equation}{section}

\def\boxit#1{\vbox{\hrule height1pt \hbox{\vrule width1pt\kern1pt
     #1\kern1pt\vrule width1pt}\hrule height1pt }}
 \def\lab#1{\boxit{\small #1}\label{#1}}
  \def\mref#1{\boxit{\small #1}\ref{#1}}
 \def\meqref#1{\boxit{\small #1}\eqref{#1}}

   \def\lab#1{\label{#1}} \def\mref#1{\ref{#1}} \def\meqref#1{\eqref{#1}}

\newcommand{\disp}{\displaystyle}

\newcommand{\bq}{\begin{equation}}
\newcommand{\eq}{\end{equation}}

\begin{document}

\title[symmetric mixed finite element]{
The simplest  mixed finite element method
   for linear elasticity in the symmetric formulation
   on $n$-rectangular grids
    }
\date{}

\author {Jun Hu}
\address{LMAM and School of Mathematical Sciences, Peking University,
  Beijing 100871, P. R. China.  hujun@math.pku.edu.cn}

\author {Hongying Man}
\address{Department of Mathematics, Beijing Institute of Technology,
  Beijing 100081, P. R. China. manhy@bit.edu.cn }

\author {Shangyou Zhang}
\address{Department of Mathematical Sciences, University of Delaware,
    Newark, DE 19716, USA.  szhang@udel.edu }

\thanks{The first author was supported by  the NSFC Project 11271035, and  in part by the NSFC Key Project 11031006.}

\begin{abstract}
   A family of mixed finite elements is proposed for
   solving the first order system of linear elasticity equations
   in any space dimension, where the stress field
    is approximated by symmetric finite element tensors.
    This family of elements has a perfect matching between the
      stress components and the displacement.
   The discrete spaces for
      the normal stress $\sigma_{ii}$,
     the shear stress $\sigma_{ij}$ and the displacement $u_i$ are
   $\operatorname{span}\{1,x_i\}$, $\operatorname{span}\{1,x_i,x_j\}$ and
	$\operatorname{span}\{1\}$, respectively, on rectangular grids.
   In particular,  the definition remains the same for all space dimensions.
     As a result of these choices, the theoretical analysis
     is independent of the spatial dimension as well.
     In 1D, this element is nothing else but the 1D
       Raviart-Thomas element, which is the only conforming element
       in this family.
   In 2D and higher dimensions, they are new elements but of the minimal
     degrees of freedom.
    The total degrees of freedom per element is $2$ plus $1$ in 1D,
        7 plus 2 in 2D, and  15 plus 3 in 3D.
   The previous record of the least degrees  of freedom  is,
      13 plus 4 in 2D, and   54 plus 12 in 3D, on the rectangular grid.
      These elements are the  simplest element for any space dimension.

    The  well-posedness condition and the optimal a priori error estimate of the family of finite elements
      are proved for  both pure  displacement and traction problems.
    Numerical tests in 2D and 3D are presented to show a superiority
       of the new element over others, as a superconvergence is
        surprisingly exhibited.
  \vskip 15pt


 \vskip 15pt


\end{abstract}
\maketitle

\section{Introduction}
The first order system of equations,  for the symmetric stress field
   $\sigma\in\Sigma:=H({\rm div},\Omega,\mathbb {S})$
 and the displacement field $u\in V:=L^2(\Omega,\mathbb{R}^n)$, reads:
Find $(\sigma,u)\in\Sigma\times V$  such that
\an{\lab{eqn1} \ad{
  (A\sigma,\tau)+({\rm div}\tau,u)&= 0 &&\forall \tau\in\Sigma,\\
   ({\rm div}\sigma,v)&= (f,v) &\qquad& \forall  v\in V.}
   }
Here the symmetric tensor-valued stress space $\Sigma$ and
   the vector-valued  displacement
   space $V$ are, respectively,
  \an{ \lab{S} H({\rm div},\Omega,\mathbb {S})
	&= \Big\{ \p{\sigma_{ij} }_{n \times n} \in H(\d, \Omega)
	\ \Big| \ \sigma_{ij} = \sigma_{ji} \  \Big\}, \\
	\lab{V}  L^2(\Omega,\mathbb{R}^n) &=
	 \Big\{ \p{u_1, & \cdots, u_n}^T
          \ \Big| \ u_i \in L^2(\Omega) \Big\}  .}
In 1D, one  example of the problem \eqref{eqn1} is the mixed formulation
   of the 1D Poisson equation;
In 2D and 3D,  the stress-displacement formulation based on the
    Hellinger-Reissner  principle for the
 linear elasticity can be regarded as a celebrated example
    of \eqref{eqn1}.

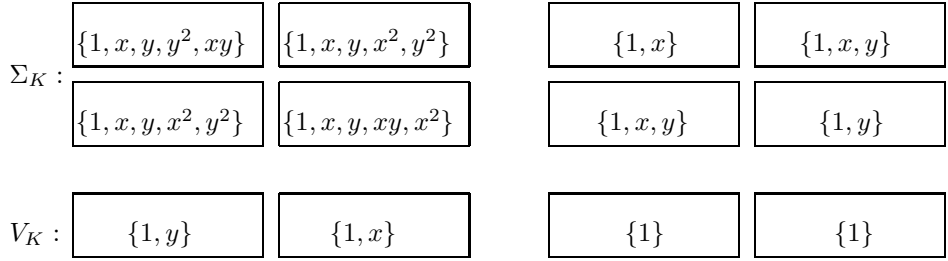
\begin{figure}[htb]
   \begin{center}\setlength{\unitlength}{1.2pt}
  \def\bo{\begin{picture}(80,80)(0,0)\put(0,0){\line(1,0){60}}
  \put(0,0){\line(0,1){20}} \put(60,20){\line(0,-1){20}}
  \put(60,20){\line(-1,0){60}}\end{picture}}
  \def\boo{\begin{picture}(80,80)(0,50)\put(0,20){\line(1,0){70}}
  \put(0,50){\line(0,1){30}} \put(70,50){\line(0,-1){30}}
  \put(70,50){\line(-1,0){70}}  \end{picture}}
 \begin{picture}(300,75)(-20,5)
  \put(0,35){\begin{picture}(80,80)(0,0)
    \put(-20,20){$\Sigma_K:$}
    \multiput(0,0)(0,25){2}{{\multiput(0,0)(65,0){2}\bo}}
      \put(1,30){$\{1,x,y,y^2,xy\}$}
     \put(66,5){$\{1,x,y, xy, x^2\}$}
     \put(1,5){$\{1,x,y,x^2,y^2\}$}
     \put(66,30){$\{1,x,y,x^2,y^2\}$}
   \put(150,0){\begin{picture}(80,80)(0,0)
     \multiput(0,0)(0,25){2}{{\multiput(0,0)(65,0){2}\bo}}
     \put(20,30){$\{1,x\}$}\put(15,5){$\{1,x,y\}$}
    \put(79,30){$\{1,x,y\}$}\put(85,5){$\{1,y\}$}  \end{picture}}
      \end{picture}}
 \put(0,0){\begin{picture}(80,80)(0,0)\put(-20,5){$V_K:$}
      {\multiput(0,0)(65,0){2}\bo}
     \put(17,5){$\{1,y\}$}\put(81,5){$\{1,x\}$}
   \put(150,0){\begin{picture}(80,80)(0,0)
      {\multiput(0,0)(65,0){2}\bo}
     \put(24,5){$\{1\}$}\put(90,5){$\{1\}$}  \end{picture}} \end{picture}}
  \end{picture}
   \end{center}
\caption{2D elements by Hu-Shi
	\cite{Hu-Shi}, Yi \cite{Yi} and this paper. }
    \lab{comparison}
\end{figure}

Because of the symmetry constraint on the stress tensor, $\sigma_{ij}
	=\sigma_{ji}$,
   it is extremely difficult to construct stable conforming finite
    elements of \eqref{eqn1} even if for
    2D and 3D, as stated in the plenary presentation to the $2002$
    International Congress of Mathematicians by  D. N.
      Arnold.
Hence compromised works use composite
    elements \cite{Arnold-Douglas-Gupta,Johnson-Mercier},
   or enforce the symmetry condition weakly
   \cite{Amara-Thomas, Arnold-Brezzi-Douglas, Boffi-Brezzi-Fortin,
        Morley, Stenberg-1, Stenberg-2, Stenberg-3}.
The landmarks in this direction are the respective works of Arnold and Winther
   \cite{Arnold-Winther-conforming} and  Arnold, Awanou, and Winther
   \cite{Arnold-Awanou-Winther}.
In particular,  a sufficient condition of the discrete stable method
    is proposed in these two papers,
   which states  that a discrete exact sequence guarantees
    the stability of  the mixed method.
Based on such a condition,  conforming mixed finite elements on
     the simplicial and rectangular triangulations are
    developed for both 2D and 3D
   \cite{Adams-Cockburn,Arnold-Awanou,Arnold-Awanou-Winther,
     Arnold-Winther-conforming}.
 In order to keep conformity the vertex degrees of freedom
      are in particular employed in these conforming methods.
To avoid the complexity of conforming mixed element and
    also vertex degrees of freedom, new weak-symmetry finite elements
   \cite{Arnold-Falk-Winther, Cockburn-Gopalakrishnan-Guzman,
    Gopalakrishnan-Guzman, Guzman},
    non-conforming finite elements
    \cite{Arnold-Winther-n, Hu-Shi, Gopalakrishnan-Guzman-n,
     Man-Hu-Shi, Yi-3D, Yi}
      are constructed.
See also \cite{Chen-Wang,Awanou} for the enrichment of
    nonconforming elements of \cite{Hu-Shi,Man-Hu-Shi}
     to conforming elements.
 However, most of these elements are difficult to be implemented;
      numerical implementation can
    only be found in \cite{Carstensen,Carstensen-Gunther-Reininghaus-Thiele2008,Yi} so far, all in 2D.

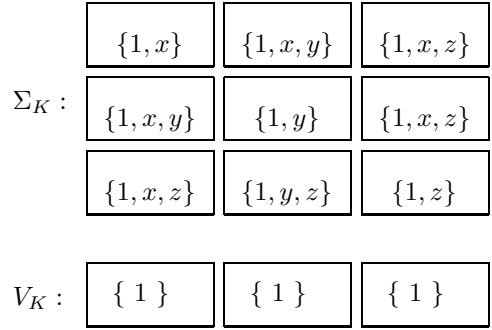
\begin{figure}[htb]
   \begin{center}\setlength{\unitlength}{.8pt}
  \def\bo{\begin{picture}(60,60)(0,0)\put(0,0){\line(1,0){60}}
  \put(0,0){\line(0,1){30}} \put(60,30){\line(0,-1){30}}
  \put(60,30){\line(-1,0){60}}\end{picture}}
 \begin{picture}(200,140)(0,-40)
   \put(0,-53){\begin{picture}(60,60)(0,0)
    \multiput(0,0)(65,0){3}{\bo}
    \multiput(12,12)(65,0){3}{\{ 1 \}}
         \put(-35, 10){$V_K:$}
     \end{picture}}
 \put(0,0){\begin{picture}(60,60)(0,0)
     \multiput(0,0)(0,35){3}{{\multiput(0,0)(65,0){3}\bo}}
     \put(-35,50){$\Sigma_K:$}
    \put(14,77){$\{1,x\}$}\put(8,42){$\{1,x,y\}$}\put(8,7){$\{1,x,z\}$}
    \put(73,77){$\{1,x,y\}$}\put(79,42){$\{1,y\}$}\put(73,7){$\{1,y,z\}$}
    \put(138,77){$\{1,x,z\}$}\put(138,42){$\{1,x,z\}$}\put(144,7){$\{1,z\}$}
    \end{picture}}
  \end{picture}
   \end{center}
 \caption{The 3D element of this paper. }
    \lab{three}
\end{figure}

In this paper,   a new family of  minimal,  any space-dimensional,
   symmetric, nonconforming mixed finite elements
     for the problem \meqref{eqn1} is constructed.
It is motivated by a simple fact that, by \meqref{S},
    the derivative on a normal   stress component $\sigma_{ii}$ is only  in $x_i$ direction;
    while those on $\sigma_{ij}$ are only in  $x_i$ and $x_j$ directions. Thus, the minimal finite element space for $\sigma_{ii}$ would
   be $\text{span}\{1, x_i\}$ on each $n$-dimensional rectangular element;
    the  minimal finite element space for $\sigma_{ij}$ would  be
     $\text{span}\{1, x_i, x_j\}$ on each $n$-dimensional rectangular element.
For the displacement \meqref{V},  there is no derivative and
    the minimal finite element space would be the constant space
      $\text{span}\{ 1 \}$. The spaces are displayed in the right diagram
    in Figure \mref{comparison} and in Figure \mref{three}.
 Surprisingly, it is shown  that these minimal finite element spaces
 can actually form a family of  stable and convergent methods
 for  \meqref{eqn1}. However, the  analysis herein has to overcome the difficulty to prove 
 the  discrete inf-sup condition,  one key ingredient for the stability 
analysis of the mixed finite element method \cite{Brezzi-Fortin}, and the difficulty related to
nonconformity of the discrete   spaces  for the stresses.  For both the elasticity problem and the Poisson problem,  the stability analysis
 of mixed finite element methods in literature  is  established by special commuting properties 
 of canonical interpolation operators defined by degrees of freedom
 of  discrete stress spaces,  see,  for instance, \cite{Adams-Cockburn,Arnold-Awanou,Arnold-Awanou-Winther,
     Arnold-Winther-conforming} and \cite{Brezzi-Fortin}.
 To overcome the first  difficulty,  a new macro-element technique is proposed to prove a  Fortin Lemma
 for mixed methods under consideration.  Note that the  macro-element technique is widely used to analyze
 the stability of mixed methods for the Stokes problem, see \cite{Brezzi-Fortin} and  references therein.
 However, it is not used to the elasticity problem before.
 For  the pure displacement problem, an explicit constructive proof is also given for the discrete inf-sup condition.
   In order to deal with the second difficulty,  a superconvergence property of the consistency  error is proved.
The mathematical {\em elegance and  beauty } of this family of  minimal elements  is gestated within,  besides
    the {\em perfect matching},  the independence of the spatial dimension $n$.
In $n$ dimension, the constructive proof of the discrete inf-sup condition
  can be divided into $n$ steps of that for the 1D  Raviart-Thomas element,  and  the consistency error
  can be decomposed as $n$ two-dimensional  consistency errors
   (For 1D, there is no consistency error.)

The superiority of the family of elements over the existing elements
   in the literature is its simplicity and high accuracy.
In fact, a family of   2D rectangular, conforming elements, of which
   the lowest order has   45 stress and 12 displacement degrees of
    freedom per element, is proposed in \cite{Arnold-Awanou}.
A nonconforming mixed finite element  based on rectangular grids
    is proposed with 19 stress and 6
    displacement degrees of freedom on each element in  \cite{Yi-3D}.
Later on, a simplified mixed
   finite element on 2D rectangular grids is constructed with 13 stress
   and 4 displacement degrees of freedom on each rectangle independently
    in \cite{Hu-Shi, Yi},
   see the left diagram in Figure \mref{comparison}, which is the simplest rectangular element of first order in 2D in the literature so far.
Doubtless, the 2D element with 7 stress and 2 displacement degrees
    of freedom on each rectangle of this paper
     is  the simplest  rectangular element,
     see the right diagram in Figure \mref{comparison}.
Due to a perfect matching (for symmetry constraint),  the new element
    has much less degrees of
    freedom (dof) but a higher order of approximation property,
    compared to previous elements  \cite{Hu-Shi, Yi-3D, Yi}.
This is confirmed by numerical results.
In 3D, the new element has only 15 stress plus 3 displacement dof on
   each element, much
   simpler than the first order element, with 54 plus 12 dof per element,
      of \cite{Man-Hu-Shi}.
Notice  that the element of
   \cite{Man-Hu-Shi} is previously
     the simplest rectangular element in 3D.

The rest of the paper is organized as follows.  The minimal element in 2D is introduced in Section 2.
 The well-posedness of the finite
    element problem, i.e. the discrete coerciveness and the
    discrete inf-sup condition,  is proved in Section 3 for the pure displacement problem.
The optimal order convergence  is shown in Section 4.
The element is extended to any space-dimension in Section 5.
In Section 6, the stability of the minimal element is shown
   for the pure traction problem.
Numerical results in 2D and 3D,
  including that for a pure traction problem,
 are provided in Section \ref{s-numerical}, which show
   a superconvergence of the minimal elements herein.

\section{A minimal element in 2D}

The 2D element is presented separately in this section
 for fixing the main idea while the whole family will
  be developed  in
 Section 5.  Also for simplicity we consider a pure
   displacement problem first.
  The analysis for other boundary value problems will
   be given in Section 6.

Consider a pure displacement problem
  (and a pure traction problem in Section 6):
\begin{subequations}
 \an{\lab{e-d}
    \d (A^{-1} \epsilon (u) ) &= f  \quad \hbox{ in } \ \Omega,\\
     u &=0 \quad\hbox{ on } \ \Gamma=\partial\Omega,  }
 \end{subequations}
The domain is assumed to be a rectangle (it is straightforward that results can be extended to domains which can be covered
 by rectangles),  which is subdivided by a family
   of rectangular grids  $\mathcal{T}_h$ (with grid size $h$).

The set of all edges in $\mathcal{T}_h$ is denoted by $\mathcal{E}_h$,
   which is divided into two sets, the set $\mathcal{E}_{h,H}$
of horizontal edges   and the set  $\mathcal{E}_{h,V}$ of vertical edges.
Given any edge $e\in\mathcal{E}_h$,  one fixed unit normal vector
   $n$ with components $(n_1,n_2)$ is assigned.
For each $K\in\mathcal{T}_h$, define  the affine invertible
   transformation
\a{
   F_K   : \qquad   \hat{K} & \rightarrow K,\\
	   \p{\hat x\\ \hat y} & \rightarrow \p{x\\y}=\p{
	   \frac{h_{x,K}}{2}\hat x+x_{0,K}\\
       \frac{h_{y,K}}{2}\hat y+y_{0,K}},
 }
with the center $(x_{0,K},y_{0,K})$ of $K$, the horizontal length $h_{x,K}$,
   and the vertical length
 $h_{y,K}$,  and the reference element $\hat{K}=[-1,1]^2$.

On each element $K\in\mathcal{T}_h$, a
	constant finite element space
   for the displacement is defined by
 \an{\lab{VK} V_K=\mathcal{P}_0(K,\mathbb{R}^2)=\Big\{ \p{v_1\\v_2} \ \big| \
	  v_1,v_2\in P_0(K) \Big\}; }
while the symmetric linear finite element space for the stress is defined by
  \an{\lab{SK}
        \Sigma_K  =\Big\{ \sigma\in\p{ P_{1,1}(K) & \mathcal{P}_1(K)\\
                  \mathcal{P}_1(K) & P_{1,2}(K)}_\mathbb{S} \Big\}, }
where subscript $\mathbb{S}$ indicates a symmetric matrix stress, and
\a{  P_{1,1}(K) & ={\rm span}\{1,x\},\\
     \mathcal{P}_1(K) & ={\rm span}\{1,x,y\}, \\
     P_{1,2}(K) &={\rm span}\{1,y\} . }
The dimension of the space $V_K$ is 2, and that
	 of $\Sigma_K$ is 7.
The nodal degrees of freedom for  $(v_1, v_2)$, $\sigma_{11}$, and $\sigma_{22}$, are

\begin{itemize}

\item the moment of degree 0 on $K$ for $v_1$ and $v_2$;

\item the moments of degree 0 on two vertical edges
   of $K$ for $\sigma_{11}$;

\item the moments of degree 0 on two horizontal edges
   of $K$ for $\sigma_{22}$;

\end{itemize}

The nodal degrees of freedom for $\sigma_{12}$ will be studied as follows.
Locally $\mathcal{P}_1(K)$ is the space of linear polynomials.
Globally,  let $W_h$ be the $P_1$-nonconforming space on $\mathcal{T}_h$,
   which is first introduced  in \cite{Park-Sheen} as a nonconforming
   approximation space to $H^1(\Omega)$ on the quadrilateral mesh;
   see also \cite{Hu-Shi05}.
To be exact,  $W_h$ is the space of piecewise linear polynomials, which
   are continuous at all mid-edge points of triangulation $\mathcal{T}_h$.
$W_h$ is the finite element space approximating function $\sigma_{12}$.

The global spaces $\Sigma_h$ and $V_h$ are defined by
\an{\lab{Sh} \Sigma_h&=\{~\sigma=\p{\sigma_{11}& \sigma_{12}\\
   \sigma_{12} & \sigma_{22}
     }\in L^2(\Omega,\mathbb{S}) \ | \ \sigma|_K\in\Sigma_K
            \ \hbox{ for all } K\in\mathcal{T}_h , \\
     &\qquad \hbox{ $\sigma_{11}$ is continuous on all  vertical interior
	 edges,  } \nonumber \\
     &\qquad \hbox{ $\sigma_{22}$ is continuous on all  horizontal interior
	 edges,  }  \nonumber \\
     &\qquad \hbox{ $\sigma_{12}$ is continuous at all
        mid-points of interior edges  }  \},
       \nonumber  \\
  \lab{Vh}  V_h &= \{v\in L^2(\Omega,\mathbb{R}^2) \ | \
         v|_K\in V(K) \ \hbox{ for all } K\in\mathcal{T}_h \}.
    }
Since $\sigma_{11}$ is continuous on all vertical interior edges,
  the derivative  $\partial_x \sigma_{11}$ is well-defined in $L^2(\Omega)$.
However,  $\sigma_{12}$ is not continuous on $\Omega$ so that
     $\partial_x \sigma_{12}$ and $\partial_y \sigma_{12}$ are not
         in $L^2(\Omega)$.
Therefore the discrete stress space $\Sigma_h$ is a nonconforming
   approximation to  $H({\rm div},\Omega,\mathbb {S})$.
So  the discrete divergence operator ${\rm div}_h$ is defined elementwise
   with respect
    to $\mathcal{T}_h$,
\a{
      \d_h\tau|_K={\rm div}(\tau|_K)\quad \forall\tau\in\Sigma_h.
  }

The mixed variational form for \meqref{e-d} is  \meqref{eqn1}.
The mixed finite element approximation of Problem \meqref{eqn1} reads: Find
   $(\sigma_h,~u_h)\in\Sigma_h\times V_h$ such that
 \e{ \left\{ \ad{
    (A\sigma_h, \tau)+({\rm div}_h\tau, u_h)&= 0 &&
             \forall \tau \in\Sigma_h,\\
     (\d_h\sigma_h, v)& = (f, v) && \forall v\in V_h.
      } \right. \lab{DP}
    }
    It follows from the definition of $\Sigma_K$ that
    $\d_h \tau_h$ are  piecewise constant for any $\tau_h\in \Sigma_h$,
    which leads to
	\a{ \d_h  \Sigma_h \subset V_h.}
This, in turn, leads to a strong discrete divergence-free space:
 \an{ \lab {kernel}
    Z_h&= \{\tau_h\in\Sigma_h \ | \ (\d_h\tau_h, v)=0 \quad \forall
                  v\in V_h\}\\
 	\nonumber
          &= \{\tau_h \in\Sigma_h \ | \  \d_h \tau_h=0
	\hbox{\ pointwise } \}.
    }
For the  analysis, define the following broken norm:
 \an{\lab{h-norm}
      \|\tau\|_{H(\d_h)}=(\|\tau\|_0^2+\|{\rm
         div}_h\tau\|_0^2)^{1/2}\quad
	\forall \tau \in \Sigma_h.
   }

The rest of this section is devoted to an alternative definition to $W_h$, the space for
  $\sigma_{12}$ in $\Sigma_h$.
The dimension of  the space $\mathcal{P}_1(K)$ is three,
    less than the number of edges or vertexes of element $K$.
The discrete shear stress $\sigma_{12}$ is still  defined by
    four vertex-value functionals, which are not linearly independent though.
A constraint can be posed on those four functionals if one defines
    a functional set $\mathcal N$ on $\mathcal{P}_1(K)$,
     cf. \cite[Lemma 2.1]{Park-Sheen}.

Here the idea from \cite{Hu-Shi05} of a frame for
    $\mathcal{P}_1(K)$ will be used. To this end,  define the frame for the space
$\mathcal{P}_1(\hat{K})  ={\rm span}\{1,\hat{x},\hat{y}\}$ by
\a{  \phi_{-1,-1}& =\frac{ 1-\hat{x}-\hat{y}} 4 , &
        \phi_{1,-1}& =\frac{  1+\hat{x}-\hat{y}} 4, \\
       \phi_{1,1}& = \frac{ 1+\hat{x}+\hat{y}} 4, &
        \phi_{-1,1}& =\frac{  1-\hat{x}+\hat{y}} 4. }
This frame is depicted in Figure \ref{basis}.

\begin{figure}[htb]
   \begin{center}\setlength{\unitlength}{1.2pt}
   \begin{picture}(70,60)(0,0)\put(0,0){\line(1,0){60}}
  \put(0,0){\line(0,1){60}} \put(60,60){\line(0,-1){60}}
   \multiput(-3,30)(12,0){6}{\line(1,0){9}}
   \multiput(30,-3)(0,12){6}{\line(0,1){9}}
  \put(60,60){\line(-1,0){60}}\put(60,60){\circle*{3}}
 \put(60,0){\circle*{3}}\put(0,60){\circle*{3}}\put(0,0){\circle*{3}}
  \put(-80,4){$\disp\phi_{-1,-1}=\frac{1-\hat{x}-\hat{y}}4$}
  \put(-80,48){$\disp\phi_{-1,1}=\frac{1-\hat{x}+\hat{y}}4$}
  \put(64,4){$\disp\phi_{1,-1}=\frac{1+\hat{x}-\hat{y}}4$}
  \put(64,48){$\disp\phi_{1,1}=\frac{1+\hat{x}+\hat{y}}4$}
  \end{picture}
   \end{center}
 \caption{Four nodal (frame/not basis) functions of $W_h$ on $\hat K$. }
    \lab{basis}
\end{figure}
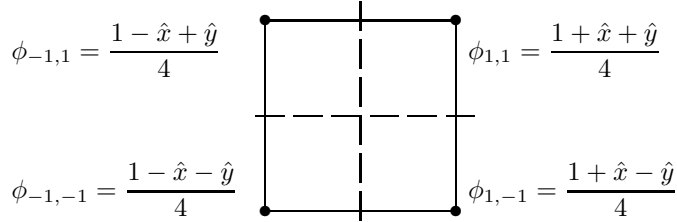

An interpolation operator $\Pi_{12}$, from $H^2(\Omega)$
	(i.e., some continuous functions) to $W_h$ is needed.
 The interpolation on $\hat K$ is defined as
  \a{\hat   \Pi_{12} \hat \sigma_{12} & =
     \hat  \sigma_{12}(\hat x_{1,\hat K}, \hat y_{1,\hat K})\phi_{-1,-1}
         + \hat  \sigma_{12}(\hat x_{2,\hat K}, \hat y_{2,\hat K})\phi_{1,-1}
      \\ &\quad \
    +\hat \sigma_{12}(\hat x_{3,\hat K}, \hat y_{3,\hat K})\phi_{1,1}
          +\hat \sigma_{12}(\hat x_{4,\hat K}, \hat y_{4,\hat K})\phi_{-1,1}, }
  where the four vertexes are numbered counterclock wise,
    \a{ (\hat x_{1,\hat K}, \hat y_{1,\hat K})&=(-1,-1), \\
         (\hat x_{2,\hat K}, \hat y_{2,\hat K})&=( 1,-1), \\
         (\hat x_{3,\hat K}, \hat y_{3,\hat K})&=( 1, 1), \\
       (\hat x_{4,\hat K}, \hat y_{4,\hat K})&=( -1, 1). }
 In the same fashion,  the interpolation $\Pi_{12}$ is defined on
      all $K\in\mathcal{T}_h$ by
  \an{ \lab{Pi12}
      \Pi_{12} \sigma_{12} (x,y) &= \sigma_{12}(x_{1,K},y_{1,K})
          \phi_{-1,-1}(F_K^{-1}(x,y))
       \\ &\quad \ \nonumber
     +  \sigma_{12}(x_{2,K},y_{2,K})
          \phi_{ 1,-1}(F_K^{-1}(x,y))
       \\ &\quad \   \nonumber
     +  \sigma_{12}(x_{3,K},y_{3,K})
          \phi_{ 1, 1}(F_K^{-1}(x,y))
       \\ &\quad \  \nonumber
     +  \sigma_{12}(x_{4,K},y_{4,K})
          \phi_{-1,1}(F_K^{-1}(x,y)), }
where $(x,y)\in K$, and $(x_{i,K},y_{i,K})$ are the four vertexes of $K$.
As $\phi_{-1,-1}(0,-1)=\phi_{1,-1}(0,-1)=1/2$, it follows that
   \a{ \Pi_{12}\sigma_{12}|_{e_{\pm}}(e_m) =
	\frac 12 \left(\sigma_{12}(e_1)+\sigma_{12}(e_2)\right), }
  where $e_+$ and $e_-$ are two sides of an edge $e\in \mathcal{E}_h$,
   $e_m$ is  the mid-point of $e$, and $e_1$ and $e_2$ are two endpoints
    of $e$.
That is,  $\Pi_{12}\sigma_{12}$ is continuous at all mid-points of edges.
For a vertex in $\mathcal{T}_h$,
  \a{  c_{i,j}=(ih,jh), \quad 0\le i, j \le N, \ N=1/h, }
  it may be shared by one, or two, or four elements $K\in\mathcal{T}_h$.
The combination of the frame functions at the vertex $c_{i,j}$ forms one
   global frame function $\phi_{i,j}$.
For example, at vertex $c_{0,1}$, as it is shared by
    two elements, $K_{1,1}=[0,h]\times[0,h]$ and $K_{1,2}=[0,h]\times[h,2h]$,
\a{ \psi_{0,1} = \begin{cases} \phi_{-1,1}(\frac 2h(x-\frac h2),
		\frac 2h(y-\frac h2)) & (x,y)\in K_{1,1}, \\
	     \phi_{-1,-1}(\frac 2h(x-\frac h2),
		\frac 2h(y-\frac {3h}2)) & (x,y)\in K_{1,2}, \\
	      0 &  \hbox{elsewhere on }\ \Omega. \end{cases} }
Note that $\psi_{i,j}$ is not continuous at $c_{i,j}$.
Thus, the finite element space for $\sigma_{12}$ in \meqref{Sh} is
   \an{\lab{Wh} W_h =
	\{ s \in L^2(\Omega) \ \mid \	 s=\sum_{i,j=0}^N
	p_{ij} \psi_{i,j} \}. }

\section{Well-posedness of the discrete problem in 2D}
This section considers the well-posedness of the discrete problem
	\meqref{DP}, which needs
   the following two conditions.

\begin{enumerate}
\item K-ellipticity. There exists a constant $C>0$, independent of the
   meshsize $h$ such that
	\an{ \lab{below} (A\tau, \tau)\geq C\|\tau\|_{H(\d_h)}^2\quad
	   \forall \tau \in Z_h, }
    where $Z_h$ is the divergence-free space defined in \meqref{kernel}.

\item  Discrete B-B condition.
    There exists a positive constant $C>0$
            independent of the meshsize $h$, such that
	\an{\lab{inf-sup}
   \inf_{v\in V_h}   \sup_{\tau\in\Sigma_h}\frac{({\rm
        div}_h\tau, v_h)}{\|\tau\|_{H(\d_h)}  \|v\|_{0} }\geq
	C .}
\end{enumerate}

\begin{theorem}
 For the discrete problem (\ref{DP}), the K-ellipticity \meqref{below}
	and the discrete B-B
 condition \meqref{inf-sup} hold uniformly.
  Consequently,  the discrete
     mixed problem \meqref{DP} has a unique solution
         $(\sigma_h,~u_h)\in\Sigma_h\times V_h$.
\end{theorem}

\begin{proof} It follows from (\ref{kernel}) that for all
    $\tau\in Z_h$, $\d_h\tau=0$.
  Thus $\|\d_h\tau\|_0=0$ and $\|\tau\|_{H(\d_h)}= \|\tau\|_0$.
  Since the operator $A$ is symmetric and positive definite, the
    $K$-ellipticity of the
             bilinear form $(A\tau,\tau)$ follows.

It remains to show the discrete B-B condition \meqref{inf-sup}.
   Since the usual technique based on canonical interpolations operators for discrete
    stress spaces  \cite{Arnold-Awanou-Winther,Arnold-Winther-conforming}
   is inapplicable here,  a constructive proof is adopted.
For convenience,  suppose that the domain $\Omega$
     is a unit square $[0,1]^2$ which  is triangulated evenly into
    $N^2$ elements, $\{ K_{i j} \}$.
For any $v\in V_h$, it can be decomposed as a sum,
\an{\lab{v-v-h}
   v_h=\sum\limits_{i=1}^{N}\sum\limits_{j=1}^NV_{ij}\varphi_{ij}(x,y),
}
where $\varphi_{ij}(x)$ is the characteristic function
 on the element $K_{ij}$,  and $V_{ij}= (V_{1,ij},$ $  V_{2,ij})$
$ = (v_h|_{K_{ij}})$.
A discrete stress function $\tau_h\in\Sigma_h$  will be
    constructed with
\a{
  \d_h\tau_h=v_h\text{ and }\|\tau_h\|_{H(\d_h)}\leq C\|v_h\|_0.
      }
The construction of $\tau_h$ is motivated by a simple
    proof of the inf-sup condition of the 1D Raviart-Thomas element
    for the 1D Poisson problem.
The shear stress $\tau_{12}$ can be taken zero,  i.e., $\tau_{12}\equiv 0$;
   the normal stress $\tau_{11}$ (resp. $\tau_{22}$ ) of $\tau_h$
   can be constructed so that it is independent of the second (resp. first)
   component of $v_h$.
In addition, $\tau_{11}$ (resp. $\tau_{22}$) can be a continuous
   piecewise linear function of the variable $x$ (resp. $y$)
  and a piecewise constant function of  $y$ (resp. $x$).
 Therefore, they are of form
\an{\lab{tau11} \tau_{11}(x,y)&=h\sum_{m=1}^{i-1} V_{1,mj}
	+ V_{1,ij} (x-x_{i-1}), \\
     \tau_{22}(x,y)&=h \sum_{k=1}^{j-1} V_{2,ik}
	+ V_{2,ij}(y-y_{j-1}),}
for  $x_{i-1}\le x < x_{i}$ and $y_{j-1}\le y < y_{j}$	
   ($(x_i,y_j)$ is the upper-right corner vertex of square $K_{ij}$.)
Thus, define
 \a{ \tau_{h}=\p{ \tau_{11}  &  0\\
                0 & \tau_{22} }\in \Sigma_h. }
By this construction,  $\partial_x \tau_{11}=(v_h)_1$ and
	$\partial_y \tau_{22} = (v_h)_2$.
This gives
\an{\lab{equal} {\rm div}_h\tau_{h}=v_h. }
 An elementary calculation gives
\a{
   \|v_h\|_0^2 &=\sum_{i,j=1}^N \|V_{ij}\varphi_{ij}\|_{0, K_{ij}}^2
    = \sum_{i,j=1}^N \int_{K_{ij}}|V_{ij}\varphi_{ij}|^2\,dxdy
   \\ &   =\sum_{i,j=1}^N ((V_{1,ij})^2+(V_{2,ij})^2)h^2.
}
By the Schwarz inequality,
 \a{ \|\tau_{11} \|_0^2&= \sum_{i,j=1}^N
      \int_{K_{ij}} \left(h\sum_{m=1}^{i-1} V_{1,mj}
	+ V_{1,ij} (x-x_{i-1})\right)^2\,dxd y \\
   &\le \sum_{i,j=1}^N
      \int_{K_{ij}} \left(h^2\sum_{m=1}^{i-1} (V_{1,mj})
	+ (V_{1,ij}) (x-x_{i-1})^2\right)\cdot i \, dxd y. }
 Further, since $N=1/h$ and $\int_{K_{ij}}=h^2$, \a{ \|\tau_{11} \|_0^2
    & \le  \sum_{i,j=1}^N
       \left(h^2\sum_{m=1}^{i} (V_{1,mj})^2 \right)\cdot N h^2
     \le  \sum_{j=1}^N
       \left(h^2\sum_{m=1}^{N} (V_{1,mj})^2 \right)\cdot N^2 h^2 \\
    &=  h^2 \sum_{i,j=1}^N (V_{1,ij})^2. }
A similar argument leads to
\[
\|\tau_{22}\|_0^2\leq h^2 \sum_{i,j=1}^N (V_{2,ij})^2.
\]
The combination of the aforementioned two identities and two inequalities yields
    \a{ \|\tau_h\|_{H({\rm div}_h)}^2
      &=\|\tau_h\|_0^2+\|{\rm div}_h\tau_h\|_0^2 \\
      &=\|\tau_{11}\|_0^2+\|\tau_{22}\|_0^2+ \|v_h\|_0^2 \le
       2\|v_h\|_0^2. }
Hence, for any $v_h\in V_h$, the B-B condition \meqref{inf-sup}
	 holds with $C=1/\sqrt 2$:
\a{
 \inf_{v_h\in V_h} \sup\limits_{\tau\in\Sigma_h}\frac{({\rm
    div}_h\tau,~v_h)}
     {\|\tau\|_{H({\rm div}_h)}  \|v_h\|_0 }
    \geq \inf_{v_h\in V_h} \disp\frac{\|v_h\|_0^2}
    {\sqrt{2}\|v_h\|_0^2}=\frac 1{\sqrt 2}.
   }
  This completes the proof.
\end{proof}

\section{Error analysis in 2D}
The section is devoted to the error estimate stated in Theorem  \ref{MainError}, which is based on the approximation error estimate of Theorem \ref{Approximationerror} and the consistency error estimate of Theorem \ref{consistencyError}.

In order to analyze the approximation error, for any $\tau\in H({\rm div},\Omega,\mathbb{S})\cap
    H^2(\Omega,\mathbb{S})$,  define an
interpolation
\an{\lab{interpolation} \Pi_h\sigma=\p{\Pi_{11}\sigma_{11} &
     \Pi_{12}\sigma_{12}\\
    \Pi_{12}\sigma_{12} & \Pi_{22}\sigma_{22}} \in\Sigma_h, }
where $\Pi_{11}$ and $\Pi_{22}$ are standard, satisfying, respectively,
 \an{ \lab{s11}
    \int_e\Pi_{11}\sigma_{11}ds & =\int_e\sigma_{11}ds \quad ~~{\rm
   for~any~vertical~edge}~e\in\mathcal{E}_h, \\
   \lab{s22}
   \int_e\Pi_{22}\sigma_{22}ds & =\int_e\sigma_{22}ds \quad  ~~{\rm
    for~any~horizontal~edge}~e\in\mathcal{E}_h. }

$\Pi_{12}$ is the interpolation operator
    defined in \meqref{Pi12}, from the space $H^2(\Omega)$ to $W_h$.
It is shown by Park and Sheen  \cite{Park-Sheen}  that
 \an{ \lab{IE12}|v-\Pi_{12}v|_{m,K}\leq
          Ch^{2-m}|v|_{2,K},~~~m=0,1, \quad K\in \mathcal{T}_h. }

\begin{theorem}\label{Approximationerror}
   For any $\sigma\in H^2(\Omega, \mathbb{S})$, it holds  that \a{
   \|\sigma-\Pi_h\sigma\|_0 &\leq Ch\|\sigma\|_1, \\ \|{\rm
   div_h}(\sigma-\Pi_h\sigma)\|_0 &\leq Ch\|\sigma\|_2. }
   \end{theorem}
\begin{proof}
By the scaling argument and the standard approximation theory,
   the following  two estimates will be proved
  \an{\lab{IE1}
   |\sigma_{11}-\Pi_{11}\sigma_{11}|_{0,K}  & \leq
    Ch|\sigma_{11}|_{1,K} \quad \forall K\in\mathcal{T}_h, \\
    \lab{IE2} |\frac{\partial}{\partial
    x}(\sigma_{11}-\Pi_{11}\sigma_{11})|_{0,K} & \leq
    Ch|\frac{\partial{\sigma_{11}}}{\partial
     x}|_{1,K} \quad \forall K\in\mathcal{T}_h. }
For any element $K\in\mathcal{T}_h$, by \meqref{s11} (i.e.,  the
   interpolation \meqref{s11} is equivalent to a mid-point interpolation),
\a{
   \|\sigma_{11}-\Pi_{11}\sigma_{11}\|_{0,K}^2&=
    \frac{ h_xh_y} 4 \int_{\hat{K}}|\hat{\sigma}_{11}-\hat{\Pi}_{11}
     \hat{\sigma}_{11}|^2d\hat{x}d\hat{y}\\
  &\le Ch^2|\hat{\sigma}_{11}|_{1,\hat{K}}^2\le
     Ch^2|\sigma_{11}|_{1,K}^2.
   }
This is (\ref{IE1}). By the reference mapping,
\an{ \lab{IE3}
   \left\|\frac{\partial}{\partial x}(\sigma_{11}-\Pi_{11}\sigma_{11})
	\right\|_{0,K}^2&= \frac{h_y}{h_x}\int_{\hat{K}}
    |\frac{\partial}{\partial\hat{x}} (\hat{\sigma}_{11}
     -\hat{\Pi}_{11}\hat{\sigma}_{11})|^2 \, d\hat{x}d\hat{y}\\
    \nonumber & \le
      C\int_{\hat{K}}|\frac{\partial}{\partial\hat{x}}\hat{\sigma}_{11}
     -\frac{\partial}{\partial\hat{x}}\hat{\Pi}_{11}\hat{\sigma}_{11}|^2
       d\hat{x}d\hat{y}.
   }
 Now \a{ \int_{\hat{K}}\frac{\partial}{\partial\hat{x}}
     \hat{\Pi}_{11}\hat \sigma_{11}d\hat{x}d\hat{y}&=
    \int_{-1}^1\{(\hat{\Pi}_{11}\hat\sigma_{11})(1,\hat y)
      -(\hat{\Pi}_{11}\hat\sigma_{11})(-1,\hat y)\}d\hat{y}\\
   &= \int_{-1}^1\{\hat\sigma_{11}(1,\hat y)-
        \hat\sigma_{11}(-1,\hat y)\}d\hat{y}\\
   &= \int_{\hat{K}}\frac{\partial\hat{\sigma}_{11}}
             {\partial\hat{x}}\, d\hat{x}d\hat{y},
  }
This means
$\disp\frac{\partial}{\partial\hat{x}}(\hat{\Pi}_{11} \hat\sigma_{11})
   =P_{0,\hat{K}}(\frac{\partial\hat{\sigma}_{11}}{\partial\hat{x}})$,
 where $P_{0,\hat{K}}$ is the projection operator onto the constant
 space on element $\hat{K}$. A substitution  of it into (\ref{IE3}) leads to
\a{  \left\|\frac{\partial}{\partial
     x}(\sigma_{11}-\Pi_{11}\sigma_{11})\right\|_{0,K}^2
  & \leq C\left\|\frac{\partial{\hat{\sigma}_{11}}}{\partial\hat{x}}-
       P_{0,\hat{K}}
    (\frac{\partial\hat{\sigma}_{11}}{\partial\hat{x}})
      \right\|_{0,\hat{K}}^2
  \\ & \leq
     C \inf_{c\in \mathbb{R}}
   \left\|(\frac{\partial\hat{\sigma}_{11}}{\partial\hat{x}}-c )
       \right\|_{0,\hat{K}}^2.
}
By the Bramble-Hilbert Lemma,
\a{  \left\|\frac{\partial}{\partial
      x}(\sigma_{11}-\Pi_{11}\sigma_{11})\right\|_{0,K}^2\leq
     C\left|\frac{\partial\hat{\sigma}_{11}}{\partial
     \hat{x}}\right|_{1,\hat{K}}^2 \leq
      Ch^2\left|\frac{\partial\sigma_{11}}{\partial x}\right|_{1,K}^2.
 }
    This is (\ref{IE2}).

A similar argument yields
 \an{\lab{IE221}
    \|\sigma_{22}-\Pi_{22}\sigma_{22}\|_{0,K} & \leq
    Ch|\sigma_{22}|_{1,K}\quad \forall K\in\mathcal{T}_h,\\
   \lab{IE222} \left\|\frac{\partial}{\partial
     y}(\sigma_{22}-\Pi_{22}\sigma_{22})\right\|_{0,K} & \leq
     Ch\left|\frac{\partial{\sigma_{22}}}{\partial
     y}\right|_{1,K}\quad \forall K\in\mathcal{T}_h. }
  Noting that the $L^2$ norm on $\Sigma$   is
 \a{ \|\sigma \|_{0,K}^2
    = \|\sigma_{11} \|_{0,K}^2+2
	\|\sigma_{12} \|_{0,K}^2 +\|\sigma_{22}\|_{0,K}^2,
    }
A combination of the estimates (\ref{IE1}), (\ref{IE2}),
   (\ref{IE221}), (\ref{IE222}) and (\ref{IE12}),  completes the proof.
\end{proof}

\begin{theorem}\label{consistencyError}  Assume that $(\sigma, u)$ be the solution to the problem \eqref{eqn1} with $u\in H_0^1(\Omega, \mathbb{R}^2)\cap H^2(\Omega, \mathbb{R}^2)$. Then,
  \begin{equation}
  \sup\limits_{\tau_h\in\Sigma_h}\frac{(A\sigma,\tau_h)+({\rm
     div}_h\tau_h,u)}{\|\tau_h\|_{H({\rm div}_h)}}\leq Ch |u|_{2}.
  \end{equation}
\end{theorem}
\begin{proof} It follows from the first equation of \eqref{eqn1} that $A\sigma=\frac{1}{2}(\nabla u+\nabla u^T)$ for the exact
    solution $u\in H_0^1(\Omega,~\mathbb{R}^2)$.  An elementwise integration by parts gives
 \a{
   (\epsilon(u),\tau_h)=-({\rm
     div}_h\tau_h,u)+\sum\limits_{K\in\mathcal{T}_h}\int_{\partial K}\tau_h
    n\cdot uds\quad \forall\tau_h\in\Sigma_h, }
which implies
 \an{\lab{eqn602}
  (A\sigma,\tau_h)+({\rm
   div}_h\tau_h,u)&
  =\sum\limits_{K\in\mathcal{T}_h}\int_{\partial K}\tau_h n\cdot uds.
   }

\begin{figure}[htb]
   \begin{center}\setlength{\unitlength}{1.3pt}
   \begin{picture}(70,95)(0,-17)\put(0,0){\line(1,0){60}}
  \put(0,0){\line(0,1){60}} \put(60,60){\line(0,-1){60}}
   \put(25,25){$K$}
  \put(60,60){\line(-1,0){60}}\put(30,60){\circle*{3}}
 \put(30,0){\circle*{3}}\put(0,30){\circle*{3}}\put(60,30){\circle*{3}}
  \put(-60,28){$\p{-\tau_{11}\\-\tau_{12}}, e_{4,K}$}
   \put(2,28){$m_{4,K}$}
  \put(63,28){$\p{ \tau_{11}\\ \tau_{12}}, e_{2,K}$}
  \put(35,28){$m_{2,K}$}
  \put(10,-15){$\p{-\tau_{12}\\-\tau_{22}}, e_{1,K}$}
    \put(22,4){$m_{1,K}$}
   \put(10,73){$\p{ \tau_{12}\\ \tau_{22}}, e_{3,K}$}
   \put(22,47){$m_{3,K}$}
  \end{picture}
   \end{center}
\caption{ $\tau_h\cdot n$ on the four edges of element $K$,
  cf. \meqref{cons5}. }    \lab{4edge}
\end{figure}
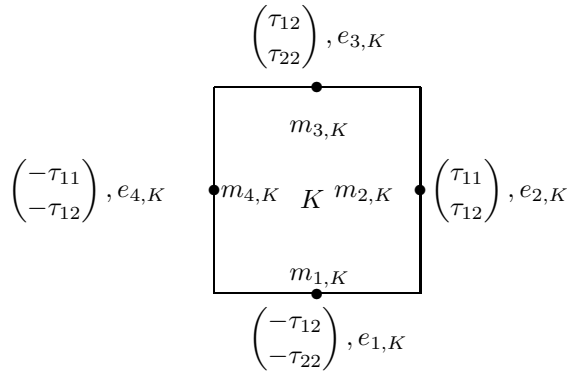

Let $\tau_h|_K=\p{\tau_{11}& \tau_{12}\\  \tau_{21}& \tau_{22}}$,
   cf. Figure \mref{4edge}.
Since $\tau_{11}$ is continuous in the $x$-direction and
  $\tau_{22}$ is continuous in the $y$-direction,
   there is a cancellation for these two components on
    the inter-element boundary.
 Since  $u\in
   H^2(\Omega,\mathbb{R}^2)\cap H^1_0(\Omega, \mathbb{R}^2)$,
  \an{\lab{cons5}
   &\quad \ \sum\limits_{K\in\mathcal{T}_h}\int_{\partial K}
             \tau_h\cdot nuds\\
   \nonumber
    &=\sum\limits_{K\in\mathcal{T}_h}
     \Big[ (\int_{e_{2,K}}-\int_{e_{4,K}}) \tau_{12}u_2ds
    + (\int_{e_{1,K}}-\int_{e_{3,K}}) \tau_{12}u_1ds\Big]. }
For any $v\in H^1(K)$, define
    the $L^2$-projection operator
     $\disp J_{e }$ on an edge $e$  by
  \a{
       J_{e }v=\frac{1}{|e |}\int_{e }vds.
   }
Because $\tau_{12}$ is continuous at the mid-point of all edges,
   it follows that, including boundary edges where $u_i=0$, on
   the horizontal edges $\mathcal{E}_{h,H}$,
\a{
  \sum\limits_{K\in\mathcal{T}_h}
     (\int_{e_{1,K}}-\int_{e_{3,K}}) \tau_{12}u_1ds
  &= \nonumber
    \sum\limits_{e\in\mathcal{E}_{h,H}}
     \int_{e } (\tau_{12}|_{e_+} - \tau_{12}|_{e_-} )u_1ds   \\
   &= \nonumber
    \sum\limits_{e\in\mathcal{E}_{h,H}}
     \int_{e } (\tau_{12}|_{e_+} - \tau_{12}|_{e_-} )(u_1-J_{e} u_1) ds.}
After inserting a same constant
    $J_K \tau_{12}=\int_K \tau_{12} dx dy / |K|$ into
    the two integrals on two horizontal edges of one element $K$,
 the sum  can be rewritten as
\an{\lab{cons1}
   &\quad \ \sum\limits_{K\in\mathcal{T}_h}
     (\int_{e_{1,K}}-\int_{e_{3,K}}) \tau_{12}u_1ds   \\
 &= \nonumber \sum\limits_{K\in\mathcal{T}_h}
      \int_{e_{1,K}}\tau_{12} (u_1-J_{e_{1,K}} u_1) ds
      -\int_{e_{3,K}} \tau_{12} (u_1-J_{e_{3,K}} u_1)   ds  \\
 &= \nonumber \sum\limits_{K\in\mathcal{T}_h}
      \int_{e_{1,K}}(\tau_{12} - J_K \tau_{12} )
	    (u_1-J_{e_{1,K}} u_1) ds \\
  &\nonumber \qquad
       -\int_{e_{3,K}} (\tau_{12}-J_{K} \tau_{12})
          (u_1-J_{e_{3,K}} u_1)   ds. }
  There is some superconvergence property for the two terms
 in (\ref{cons1}) if they are considered together.
   In fact, on the reference element  $\hat K$, $\hat\tau_{12}(\hat x, \pm 1)
    = \hat\tau_{12} (0,0) + \hat x \partial_{\hat x} \hat\tau_{12}(0,0)
       \pm \partial_{\hat y} \hat\tau_{12}(0,0)$,
	and $J_{\hat K} \hat \tau_{12} = \hat\tau_{12} (0,0)$.  The  property  of $J_e$  gives
\a{ & \quad \
      \frac 12  \int_{-1}^1  (\hat \tau_{12}- J_{\hat K} \hat \tau_{12} )
      \Big[ (\hat u_1- \hat J_{\hat e_1} \hat u_1)(\hat x, -1)
         -(\hat u_1- \hat J_{\hat e_3} \hat u_1)(\hat x, 1)\Big] d\hat x \\
    & = -\frac12 \int_{-1}^1  \hat x
            \frac{\partial}{\partial\hat x} \hat\tau_{12}
           \Big[\int_{-1}^1 \frac{\partial}{\partial \hat y }
	    \hat u_1d\hat y -
             \frac 12 \int_{-1}^1 (\hat u(\hat x,1)
	             -\hat u(\hat x,-1) ) d \hat x
	\Big] d \hat x  \\
    & = -\frac12 \int_{-1}^1  \hat x
            \frac{\partial}{\partial\hat x} \hat\tau_{12}
           \Big[\int_{-1}^1 \frac{\partial}{\partial \hat y }
	    \hat u_1d\hat y -
             \frac 12 \int_{-1}^1(\int_{-1}^1 \frac{\partial}{\partial \hat y }
               \hat u(\hat x,\hat y)d\hat y ) d \hat x
	\Big] d \hat x  \\
   & = -\frac14 \int_{-1}^1  \hat x
            \frac{\partial}{\partial\hat x} \hat\tau_{12}
           \Big[\int_{-1}^1 \int_{-1}^1 ( \frac{\partial}{\partial \hat y }
	    \hat u_1(\hat x,\hat y) -  \frac{\partial}{\partial \hat y }
               \hat u(\hat t, \hat y)) d\hat y   d \hat t
	\Big] d \hat x  \\
    & = - \frac 14\int_{-1}^1  \hat x
      \frac{\partial}{\partial\hat x} \hat\tau_{12}
           \Big[\int_{-1}^1 \int_{-1}^1 \int_{\hat t}^{\hat x}
              \frac{\partial^2}{\partial\hat x\partial \hat y }
             \hat u_1 (\hat s, \hat y)
            d \hat s d\hat y  d \hat t
	\Big] d \hat x  .  }
   By the Schwarz inequality and  \meqref{cons1},
  \a{ & \quad \
    \left| \sum\limits_{K\in\mathcal{T}_h}
     (\int_{e_{1,K}}-\int_{e_{3,K}}) \tau_{12}u_1ds   \right|^2
  \\ & = \frac {h^2}{2^2} \left| \sum\limits_{K\in\mathcal{T}_h}
	  \int_{-1}^1  \frac{\partial}{\partial\hat x} \hat\tau_{12} \hat x
           \Big[\int_{-1}^1 \int_{-1}^1 \int_{\hat t}^{\hat x}
              \frac{\partial^2}{\partial\hat x\partial \hat y }
             \hat u_1 (\hat s, \hat y)
            d \hat s d\hat y  d \hat t
	\Big] d \hat x  \right|^2
  \\   & \le C   h^2  \left(\sum\limits_{K\in\mathcal{T}_h}
         \left\|\frac{\partial}{\partial\hat x} \hat\tau_{12}
          \right\|_{0,\hat K}^2 \right) \left(
           \sum\limits_{K\in\mathcal{T}_h}
         \left\|\frac{\partial^2}{\partial\hat x\partial\hat y}
             \hat u_1 \right\|_{0,\hat K}^2 \right)   \\
    & =  C   h^2  \left( \left\|\frac{\partial}{\partial  x}  \tau_{12}
          \right\|_{0}^2 \frac{ h^2 }{2^2}
         \left\|\frac{\partial^2}{\partial  x\partial  y}
               u_1 \right\|_{0 }^2 \right) \\
	& \le C h^4 |  \tau_{12}  |_{1,h}^2   | u_1  |_{2}^2.      }
Here $|  \cdot  |_{1,h}$ is the elementwise semi-$H^1$ norm.
A similar argument bounds the other term in \meqref{cons5} by
 \a{ \left| \sum\limits_{K\in\mathcal{T}_h}
      (\int_{e_{2,K}}-\int_{e_{4,K}}) \tau_{12}u_2ds \right|
    \le C h^2  |  \tau_{12} |_{1,h}   | u_2  |_{2} .      }
A combination of these two estimates with (\ref{eqn602}) implies
\a{  |(A\sigma,\tau_h)+({\rm
     div}_h\tau_h,u)|\leq Ch^2|u|_2|\tau_h|_{1,h}.}
By the inverse inequality,
   \an{ \lab{c2}
    |(A\sigma,\tau_h)+({\rm
     div}_h\tau_h,u)|  \leq Ch|u|_2\|\tau_h\|_0.
  }
\end{proof}

\begin{theorem}\label{MainError}  Let
  $(\sigma, u)\in\Sigma\times V$ be the exact solution of
   problem \meqref{eqn1} and $(\tau_h, u_h)\in\Sigma_h\times
   V_h$ the finite element solution of \meqref{DP}.  Then
\a{   \|\sigma-\sigma_h\|_0
    &\le Ch(\|u\|_2+\|\sigma\|_2),\\
     \|{\rm div}_h(\sigma-\sigma_h)\|_0
    &\le Ch(\|u\|_2+\|\sigma\|_2),\\
      \|u-u_h\|_0&\le     Ch(\|u\|_2+\|\sigma\|_2).
      }
\end{theorem}

\begin{proof}
   Let \a{ Z_f=\{\tau\in\Sigma_h \ | \ ({\rm
    div}_h\tau,v)=(f,v)\quad \forall v\in V_h\}. }
The finite element solution $\sigma_h$ is in $Z_f$.
Thus, for any $\tau\in Z_f$, it holds $\sigma_h-\tau\in Z_h$, i.e.,
   \a{{\rm  div}_h(\sigma_h-\tau)=0. }
It follows from the $K$-ellipticity  (cf. \meqref{below})
that, for all $\tau\in Z_f$,
 \a{ C\|\sigma_h-\tau\|_0^2&\le
         (A(\sigma_h-\tau),\sigma_h-\tau)\\
   &= (A(\sigma-\tau),\sigma_h-\tau)+(A(\sigma_h-\sigma),\sigma_h-\tau)\\
    &=(A(\sigma-\tau),\sigma_h-\tau)-(A\sigma,\sigma_h-\tau)-({\rm
    div}_h(\sigma_h-\tau),u_h)\\
    &=(A(\sigma-\tau),\sigma_h-\tau)-(A\sigma,\sigma_h-\tau)
    \\
    &=(A(\sigma-\tau),\sigma_h-\tau)-(A\sigma,\sigma_h-\tau)
     -({\rm
    div}_h(\sigma_h-\tau),u).
     }
An application of the Schwarz inequality leads to
\a{  \|\sigma_h-\tau\|_{H({\rm div}_h)} &=\|\sigma_h-\tau\|_0 \\
  &\le  C \|\sigma-\tau\|_{H({\rm div}_h)}-
        \frac{(A\sigma,\sigma_h-\tau)+({\rm
    div}_h(\sigma_h-\tau),u)}{ C\|\sigma_h-\tau\|_{H({\rm div}_h)}}\\
  &=  C \|\sigma-\tau\|_{H({\rm div}_h)}+
   \sup\limits_{\tau_h\in\Sigma_h}\frac{(A\sigma, \tau_h)+({\rm
    div}_h \tau_h,u )}{C \| \tau_h\|_{H({\rm div}_h)}} .  }
By the triangle inequality,
\an{\lab {ER2}
    \|\sigma-\sigma_h\|_{H({\rm div}_h)}\leq
   C\{\inf\limits_{\tau\in Z_f}\|\sigma-\tau\|_{H({\rm div}_h)}
   +\sup\limits_{\tau_h\in\Sigma_h}\frac{(A\sigma,\tau_h)+({\rm
     div}_h\tau_h,u)}{\|\tau_h\|_{H({\rm div}_h)}}\}.
    }
For a given $\tau_h\in \Sigma_h$, the discrete B-B condition \meqref{inf-sup}
    ensures that the following problem has at least one solution
     $\gamma_h\in\Sigma_h$,
    cf. \cite{Brezzi-Fortin},
  \an{\lab {ER1}
({\rm div}_h \gamma_h,v_h)=({\rm div}_h(\sigma-\tau_h),v_h)
	\quad \forall~v_h\in V_h. }
It follows from the B-B condition \meqref{inf-sup} that
\a{
      \|\gamma_h\|_{H({\rm div}_h)}&\le
     \frac{1}{C}\sup\limits_{v_h\in V_h}\frac{({\rm
     div}_h\gamma_h,~v_h)}{\|v_h\|_0}
    = \frac{1}{C}\sup\limits_{v_h\in V_h}\frac{({\rm
       div}_h(\sigma-\tau_h,~v_h)}{\|v_h\|_0}\\
    &\le \frac{1}{C}\|{\rm div}_h(\sigma-\tau_h)\|_0. }
The identity \meqref{ER1} asserts that $\gamma_h+\tau_h\in Z_f$.
  The choice $\tau=\gamma_h+\tau_h$ in \meqref{ER2} leads to
\a{ & \quad \ \|\sigma-\sigma_h\|_{H({\rm div}_h)}\\
   &\le  C\{\|\sigma-\tau_h\|_{H({\rm div}_h)}
    +\|\gamma\|_{H({\rm div}_h)}
    +\sup\limits_{\tau_h\in\Sigma_h}\frac{(A\sigma,\tau_h)+({\rm
     div}_h\tau_h,u)}{\|\tau_h\|_{H({\rm div}_h)}} \}\\
   &\le  C\{\|\sigma-\tau_h\|_{H({\rm div}_h)}
         +\sup\limits_{\tau_h\in\Sigma_h}\frac{(A\sigma,\tau_h)+({\rm
     div}_h\tau_h,u)}{\|\tau_h\|_{H({\rm div}_h)}}  \}. }
That is,
  \an{ \lab{abstract}
    \|\sigma-\sigma_h\|_{H({\rm div}_h)}\leq
    C\{\inf\limits_{\tau_h\in\Sigma_h}\|\sigma-\tau_h\|_{H({\rm
    div}_h)}+\sup\limits_{\tau\in\Sigma_h}\frac{(A\sigma,\tau)+({\rm
     div}_h\tau,u)}{\|\tau\|_{H({\rm div}_h)}}\}. }
The first term on the right-hand side of \eqref{abstract} is the approximation error.
The choice $\tau_h=\Pi_h\sigma$ with Theorem \ref{Approximationerror}
 gives its upper bound. The second term on the right-hand side of
  \eqref{abstract} is the usual consistency error for the nonconforming finite
 element method, which has already been bounded in Theorem \ref{consistencyError}.
 A  combination of  these two theorems  implies
\a{ \|\sigma-\sigma_h\|_0 & \leq Ch(\|u\|_2+\|\sigma\|_2), \\
   \|{\rm div}_h(\sigma-\sigma_h)\|_0 & \leq Ch(\|u\|_2+\|\sigma\|_2). }

The rest of the proof is concerned with the estimation of $u-u_h$. In view of the discrete
B-B Condition \meqref{inf-sup}, it holds, for any $v\in V_h$,
\a{
 &\quad \ C\|u_h-v\|_0 \\
  &\le \sup\limits_{\tau\in\Sigma_h}\frac{({\rm
    div}_h\tau,u_h-v)}{\|\tau\|_{H({\rm div}_h )}}
   =\sup\limits_{\tau\in\Sigma_h}\frac{({\rm
     div}_h\tau,u_h-u+u-v)}{\|\tau\|_{H({\rm div}_h )}}\\
&\leq \sup\limits_{\tau\in\Sigma_h}\frac{({\rm
   div}_h\tau,u_h-u)}{\|\tau\|_{H({\rm
    div}_h )}}+\|u-v\|_0\\
&= \sup\limits_{\tau\in\Sigma_h}\frac{({\rm
div}_h\tau,u_h)+(A\sigma,\tau)-(A\sigma,\tau)-({\rm
div}_h\tau,u )}{\|\tau\|_{H({\rm div}_h )}}+\|u-v\|_0\\
&= \sup\limits_{\tau\in\Sigma_h}\frac{(A(\sigma-\sigma_h),\tau)
   -(A\sigma,\tau)-({\rm
    div}_h\tau,u )}{\|\tau\|_{H({\rm div}_h )}}+\|u-v\|_0\\
&\le \sup\limits_{\tau\in\Sigma_h}\frac{(A\sigma,\tau)+({\rm
   div}_h\tau,u )}{\|\tau\|_{H({\rm div}_h )}}+
    C(\|\sigma-\sigma_h\|_0+\|u-v\|_0).
   }
 By \meqref{c2}  and the error estimation of
   $\|\sigma-\sigma_h\|_0$,  the triangle inequality plus
    $v=P_hu$  ($P_h$ is the $L^2$ projection into piecewise constant
      spaces)  yield
\a{
\|u-u_h\|_0&\le \|u-P_hu\|_0+\|P_hu-u_h\|_0\\
&\le Ch|u|_2 + C(\|\sigma-\sigma_h\|_0+\|u-P_hu\|_0) \\
&\le Ch(\|u\|_2+\|\sigma\|_2).
}
That completes  the proof of this theorem.
\end{proof}

\section{The minimal element in any spatial dimension}
Assume the domain $\Omega$ is a unit hypercube
   $[0,1]^n$ in the
   $n$-dimensional space, which is subdivided by a uniform rectangular
     grid of $N^n$ cubes:
   \a{ {\mathcal T}_h &:=\{ K_{i_1,i_2,...,i_n}
	=[(i_1-1)h,i_1h]\times \cdots [(i_n-1)h,i_nh], \\
	 &\qquad\qquad 1\le i_1,\cdots i_n\le N; \ h=1/N\}. }
The set of all $(n-1)$-dimensional face hyperplanes of the triangulation
  $\mathcal{T}_h$ that are perpendicular to the axis $x_i$ is denoted
   by $\mathcal{E}_{n-1,i}$. That is
   \a{ \mathcal{E}_{n-1,i}&=\{ [(i_1-1)h,i_1h]\times \cdots\times
        [(i_{i-1}-1)h,i_{i-1}h]\times \{i_ih\} \\
   &\qquad \times \cdots\times
        [(i_n-1)h,i_nh], \ 1\le i_1,\cdots i_n\le N, \ 0\le i_i\le N \}. }
The internal hyperplanes are denoted by
   \a{ \mathcal{E}_{n-1,i}(\Omega) & =\mathcal{E}_{n-1,i}\cap \Omega.}
 The set of all  $(n-2)$-dimensional mid-surface hyperplanes
	(orthogonal to both $x_i$ and $x_j$ axes) are denoted  by
  \a{
  \mathcal{E}_{n-2, ij}&=\{ [(i_1-1)h,i_1h]\times \cdots\times
     \{i_ih\} \times \cdots\times  \{(i_j-\frac 12)h\}\times \cdots\times \\
      &\qquad [(i_n-1)h,i_nh],
      \  1\le i_1,\cdots i_n\le N, \ 0\le i_i\le N \}\\
   &\cup\{ [(i_1-1)h,i_1h]\times \cdots\times
   \{(i_i-\frac 12) h\} \times \cdots\times  \{ i_j h\}\times \cdots\times \\
      &\qquad [(i_n-1)h,i_nh], \  1\le i_1,\cdots i_n\le N, \ 0\le i_i\le N \}.
   }
   In addition,  define $\mathcal{E}_{n-2,ij}(K):=\mathcal{E}_{n-2, ij}
      \cap \partial K$ for any $K\in\mathcal{T}_h$.
 In 2D, these sets are
  \a{ \mathcal{E}_{1,1} &=\{ \hbox{all edges in $\mathcal{T}_h$ perpendicular
       to $x_1$ } \}, \\
      \mathcal{E}_{1,2} &=\{ \hbox{all edges in $\mathcal{T}_h$ perpendicular
       to $x_2$ } \}, \\
      \mathcal{E}_{0,12} &=\{ \hbox{all mid-points of edges in
           $\mathcal{T}_h$  } \}.}
  In 3D, they are
 \a{ \mathcal{E}_{2,i} &=\{ \hbox{all squares in $\mathcal{T}_h$ perpendicular
       to $x_i$ } \}, \ 1\le i\le 3, \\
      \mathcal{E}_{1,ij} &=\{ \hbox{all mid-square edges of squares
         in $\mathcal{E}_{2,i}$ and $\mathcal{E}_{2,j}$, }\\
         &\qquad  \hbox{  parallel  to $x_k$ } \},\
	 i\ne j\ne k\in\{1,2,3\}.}

  In $n$ space-dimension,  the symmetric tensor space is
  defined in \meqref{S}.
The discrete stress space is  defined by
   \an{\lab{S-n}
   \ad{ \Sigma_h
	& := \Big\{ \p{\tau_{ij } }_{n\times n} \in L^2(\Omega,
	\mathbb{R}^{n\times n} )
	\ \Big| \ \tau_{ij} =\tau_{ji};  \\
       &\qquad
          \tau_{ii}|_K \in \operatorname{span}\{1, x_i  \},\tau_{ii}
     \hbox{ is continuous  on  $E_i\in\mathcal{E}_{n-1,i}$;
            }   \\
            &\qquad
       \tau_{ij}|_K \in \operatorname{span}\{1, x_i, x_j  \}, \
          \tau_{ij} \hbox{ is continuous on $E_{ij}\in\mathcal{E}_{n-2, ij}
              (\Omega)$} \Big\}.
    }
    }
Some comments are in order for this family of minimal finite element spaces.
\begin{remark} The normal stress $\tau_{ii}$ is a constant on
   each $(n-1)$-dimensional hyper-plane $E_i\in\mathcal{E}_{n-1,i}$.
In addition, for the case $n=1$, $\Sigma_h$ is
$$
      \{\tau_{11}\in L^2(\Omega, \mathbb{R})\ \big | \
    \tau|_K\in\text{span}\{1, x\} \text{ is continuous at the nodes }\}
   \subset H^1(\Omega),
$$
    the 1D Raviart-Thomas space, which is the only conforming space in
   this family.
   \end{remark}

\begin{remark} The dimension of the space
\begin{equation*}
  \begin{split}
  \Sigma_{h,ij}:= \{\tau_{ij}\in L^2(\Omega, \mathbb{R})\ \big | \
      & \tau_{ij}|_K \in \operatorname{span}\{1, x_i, x_j  \}, \\
      &    \tau_{ij} \hbox{ is continuous on
          $E_{ij}\in\mathcal{E}_{n-2, ij}(\Omega)$} \}
   \end{split}
   \end{equation*}
is \a{  N^{n-2}((n+1)^2-1)=N^n+2N^{n-1}. }
    see \cite{Park-Sheen} for more details for 2D.
\end{remark}
  Let us give the local basis for $\tau_{ii}$ and but
  a local frame (not basis) for $\tau_{ij}$ on an
   element $K:=K_{i_1,i_2,\dots,i_n}\in \mathcal{T}_h$.
Define, for $(x_1,\dots,x_n)\in K$,
  \a{  \psi^{(k)}_{ii,K} (x_1,\dots,x_n) &=
         \hat \psi^{(k)}\left( \frac{ x_{i } - (i_i-1/2)h }{ h /2 } \right),
     \quad k=0,1,  }
  where \a{ \hat\psi^{(0)} (\hat x) &= \frac {1-\hat x}2, &
        \hat\psi^{(1)} (\hat x) &= \frac {1+\hat x}2,
     \qquad \hat x\in [-1,1]. }
Define, for $ k=0,1,2,3,$ for $(x_1,\dots,x_n)\in K$,
   \a{  \phi^{(k)}_{ij,K} (x_1,\dots,x_n) &=
         \hat \phi^{(k)}\left( \frac{ x_{i } - (i_i-1/2)h }{ h /2 },
         \frac{ x_{j} - (i_j-1/2)h }{ h /2 }  \right),  }
  where  (cf. Figure \mref{basis}), for $(\hat x, \hat y)\in [-1,1]^2$,
   \a{ \hat\phi^{(0)} (\hat x,\hat y) &= \frac {1-\hat x-\hat y}4, &
        \hat\phi^{(1)} (\hat x,\hat y) &= \frac {1+\hat x-\hat y}4, \\
    \hat\phi^{(2)} (\hat x,\hat y) &= \frac {1+\hat x+\hat y}4, &
        \hat\phi^{(3)} (\hat x,\hat y) &= \frac {1-\hat x +\hat y}4. }
Note that the above four functions are not linearly independent.
In fact,
  \a{ \hat\phi^{(0)}- \hat\phi^{(1)}  + \hat\phi^{(2)} -\hat\phi^{(3)}
	\equiv 0. }

Then  the finite element space can be alternatively defined by
   \an{\lab{S-n-2} \Sigma_h
	& = \Big\{ \p{\tau_{ij } }_{n\times n} \in L^2(\Omega,
	\mathbb{R}^{n\times n} )
	\ \Big| \ \tau_{ij}=\tau_{ji};   \\
     \nonumber   &\qquad
	 \tau_{ii}|_K = \sum_{k=0}^1
   \tau_{ii}(E_{n-1, i}^{(k)}(K)) \psi^{(k)}_{ii,K} (x_1,\dots,x_n);\\
        \nonumber    &\qquad
        \tau_{ij}|_K  =\sum_{k=0}^3 p_{ij}(\hat{E}^{(k)}_{n-2}(K))
     \phi^{(k)}_{ij,K} (x_1,\dots,x_n)
         \Big\}. }
Here $\tau_{ii}(\hat{E}_{n-1, i}^{(k)}(K))$ are the  values of $\tau_{ii}$
   at the centers of the $(n-1)$-dimensional hyperplanes of
	$K=K_{i_1,\cdots,i_n}$:
  \a{ \hat{E}_{n-1, i}^{(k)}(K)
	  =\p{ (i_1-\frac 12)h \\ \vdots \\
             (i_{i-1}-\frac 12)h \\ (i_i - k)h \\
	    \vdots \\ (i_{n}-\frac 12)h}, \ k=0,1; }
  $p_{ij}(\hat{E}^{(k)}_{n-2}(K))\in\mathbb{R}$ are some parameters
   associated to the center-point of four
     $(n-2)$-dimensional hyperplanes of $K$
     which are continuous on the four (two on the boundary)
    $n$-cubes sharing the  point:
    \a{ \hat{E}_{n-2}^{(k)}(K) =\p{ (i_1-\frac 12)h \\ \vdots \\
              (i_i - 0)h \\\vdots \\  (i_j - 0)h \\
	    \vdots \\ (i_{n}-\frac 12)h}, \
         \p{ (i_1-\frac 12)h \\ \vdots \\
              (i_i - 1)h \\\vdots \\  (i_j - 0)h \\
	    \vdots \\ (i_{n}-\frac 12)h},  \
         \p{ (i_1-\frac 12)h \\ \vdots \\
              (i_i - 1)h \\\vdots \\  (i_j - 1)h \\
	    \vdots \\ (i_{n}-\frac 12)h},  \
         \p{ (i_1-\frac 12)h \\ \vdots \\
              (i_i - 0)h \\\vdots \\  (i_j - 1)h \\
	    \vdots \\ (i_{n}-\frac 12)h}.}
As in 2D,  the discrete displacement space is
   \an{\lab{V-n}
      V_h =\{ v\in L^2(\Omega,\mathbb{R}^n) \ | \
	v|_K \hbox{ \ is a constant vector } \}. }
In particular,  the dof of the 3D mixed element is plotted in Figure \mref{three}.

In the $n$-dimension,  since $\d_h\Sigma_h\subset V_h$,
  the K-ellipticity \meqref{below} is proved exactly the
same way as in 2D. The explicit construction proof of
  the discrete B-B condition \meqref{inf-sup} can be divided into
 $n$ essentially 1-dimensional construction proofs
   similar to that for the 1D Raviart-Thomas element of the
   1D Poisson equation, see Section 3
  for more details for 2D. For the consistency error in \meqref{ER2},
  the proof remains the same
  except  there is a multiple summation instead of 2-index summation.
  All the analysis in 2D remains the same for $n$-D.

\section{The pure traction problem}
   This section considers the pure traction problem,   i.e.,  the stress space is subject to zero Neumann boundary
  condition while no boundary condition on the displacement.  In practice,  part of elasticity body should be located, i.e,
    the displacement has a Dirichlet boundary condition on  some non-zero measure boundary.  But the pure traction problem
   is the most difficult one in mathematical analysis.  A similar  proof for Theorem \mref{inf0}  can  prove it for partial displacement problems.
   For ease of presentation,  details are  presented only for  two dimensions.  Note that the argument in any dimension is similar.  The main idea is  to use the macro-element  technique where we  construct a mass-preserving quasi-interpolation operator.

Let $\text{RM }$ be the rigid motion space in  two dimensions, which reads
$$
\text{RM}:=\text{span}\bigg\{\p{1\\ 0}, \p{0\\ 1}, \p{y\\ -x}\bigg\}.
$$
Consider a pure traction problem:
\begin{subequations}
 \an{\lab{e0}
    \d (A^{-1} \epsilon (u) ) &= f  \quad \hbox{ in } \ \Omega=(0,1)^2,\\
     \epsilon (u) \cdot  n &=0 \quad\hbox{ on } \ \Gamma=\partial\Omega, \\
     ( u, v) &=0  \quad \forall v\in \text{RM}.}
 \end{subequations}
By the same discretization of uniform square grid $\Cal T_h$ with $h=1/N$
 as in \S2,  the finite element equations \meqref{DP} remain the same except
 the spaces are changed with boundary and rigid-motion free conditions:
\a{  \ad{ (A\sigma_h, \tau) + (\d_h \tau, u_h) & = 0 && \forall\tau \in \Sigma_{h,0}, \\
     (\d_h \sigma_h, v) &=(f,v) && \forall v \in V_{h,0}, } }
where
   \an{ \lab{S-h-0}
    \Sigma_{h,0}
     &=\{ \sigma=\p{\sigma_{11} & \sigma_{12} \\ \sigma_{12} & \sigma_{22} }\in  \Sigma_h \mid
      \sigma(m_e)\cdot n=0 \ \quad  \forall e\in(\Cal E_{h}\cap\Gamma)\}, \\
	\lab{V-h-0}
	 V_{h,0}
     &=\{ v=\p{v_1 \\v_2 }\in V_h \mid
       \ (v, w)=0 \quad \forall w\in \text{RM}
       \}.
     	}
Here $m_e$ is the mid-point of an edge $e$,  and $\Sigma_h$ and
    $V_h$ are defined in \meqref{Sh} and \meqref{Vh}, respectively.
The earlier  analysis remains the same except the discrete B-B condition \meqref{inf-sup}
 as the stress space $\Sigma_{h,0}$ is much smaller than before.

\def\sq{\begin{picture}(40,40)(0,0)
   \multiput(0,0)(0,20){3}{\line(1,0){40}}
   \multiput(0,0)(20,0){3}{\line(0,1){40}}\end{picture}}

 \begin{theorem}\lab{inf0}
  The following discrete B-B condition holds uniformly,
   \a{ \inf_{v_h\in V_{h,0}} \sup_{\sigma_h \in \Sigma_{h,0}}
   \frac{(\d_h \sigma_h, v_h) } { \|\sigma_h\|_{H(\d_h)} \|v_h\|_0 }
    \ge C. }
\end{theorem}

\begin{proof} Let  $v_h =((v_h)_1, (v_h)_2)
   = \sum (v_h)_{ij} \varphi_{ij}$  as in \meqref{v-v-h},
   where $(v_h)_{ij}$ $  = ((v_h)_{1,ij},$ $ (v_h)_{2,ij})$ is the constant value of
   $v_h$ on square $K_{ij}$.
With the boundary condition on the stress,  it is impossible to match $ (v_h)_1 $ by $\partial_x \tau_{11}$ alone as
   in \meqref{tau11}.
That is,  because the dof of $ (v_h)_1 $ is $n^2-1.5$ (due to a mixed constraint
   with the second component $ (v_h)_2$),
but the dof of $\{\tau_{11}\}$ is only
    $n(n-1)$. This indicates  that the help from $\partial_y\tau_{12}$ is indispensable.
But the traditional trick of interpolating smooth B-B stress function
    does not work here as $(\tau_{12})$ does not have enough dof.
In other words,  the support of $\tau_{21}$ is non-local,
  at least on four neighboring squares.  Given $v_h$,
   a discrete B-B stress function will be constructed  in
  two steps. First, a macro-element technique will produce a $\tilde \sigma_h$ globally
   so that $v_h - \d_h \tilde \sigma_h$ is rigid-motion free on each
    $(2\times 2)$  macro-element $K_{2i,2j,2h}:=[x_{2i}, x_{2i+2}]\times
    [y_{2j}, y_{2j+2}]$.
In a second step,  construct, macro-element by macro-element,
    a $  {\bar \sigma}_h$ locally by internal dof only,
   so that $\d_h {\bar \sigma}_h= v_h - \d_h \tilde \sigma_h$.

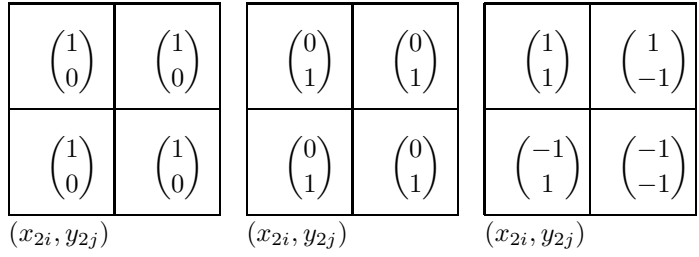
\begin{figure}[htb]
   \begin{center}\setlength{\unitlength}{2pt}
   \begin{picture}(130,40)(0,0)
  \multiput(0,0)(45,0){3}{\sq} \multiput(0,-5)(45,0){3}{$(x_{2i},y_{2j})$}
\put(0,0){\begin{picture}(60,60)(0,0)
   \put(15,29){\makebox(-6,0){$\p{1\\0}$}}   \put(35,29){\makebox(-6,0){$\p{1\\0}$}}
      \put(15,9){\makebox(-6,0){$\p{1\\0}$}} \put(35,9){\makebox(-6,0){$\p{1\\0}$}}
     \end{picture}}
\put(45,0){\begin{picture}(60,60)(0,0)
   \put(15,29){\makebox(-6,0){$\p{0\\1}$}}  \put(35,29){\makebox(-6,0){$\p{0\\1}$}}
      \put(15,9){\makebox(-6,0){$\p{0\\1}$}} \put(35,9){\makebox(-6,0){$\p{0\\1}$}}
     \end{picture}}
\put(90,0){\begin{picture}(60,60)(0,0)
    \put(15,29){\makebox(-6,0){$\p{1\\1}$}}  \put(35,29){\makebox(-6,0){$\p{1\\-1}$}}
      \put(15,9){\makebox(-6,0){$\p{-1\\1}$}} \put(35,9){\makebox(-6,0){$\p{-1\\-1}$}}
     \end{picture}}
   \end{picture}
   \end{center}
\caption{ \lab{f3}    Nodal values of three orthogonal basis functions,
   $\{\phi^r_{m,ij}, m=1,2,3\}$, of the
     rigid-motion space on macro-element $K_{2i,2j,2h}:=[x_{2i}, x_{2i+2}]\times
    [y_{2j}, y_{2j+2}]$, cf. \meqref{r3}. }
\end{figure}

To this end,  define a local rigid-motion space  on each
   macro-element $K_{2i,2j,2h}$
\an{\lab{r3} R_{ij} = \operatorname{span}
    \{ \phi^r_{1,ij},\phi^r_{2,ij},\phi^r_{3,ij}\}, }
  where $\phi^r_{1,ij}$ are defined in Figure \mref{f3}, piecewise constant
   functions. Assume $N$ is an even integer and decompose $v_h$ into two parts, a local rigid-motion and a
   global  rigid-motion-free part,
\an{ \lab{2p}
    v_h = \tilde v_h + \bar v_h,
     \quad  \tilde v_h=P_{L^2 ( R_{ij} )} v_h
         \ \hbox{ for } 0\le i,j\le N/2-1. }
Here the projection $P_{L^2 ( R_{ij} )} v_h$ is defined as
   \a{ \int_{K_{2i,2j,2h}} P_{L^2 ( R_{ij} )} v_h \cdot \phi^r_{m,ij} dx\,dy=\int_{K_{2i,2j,2h}} v_h \cdot \phi^r_{m,ij} dx\,dy,
    \quad m=1,2,3. }
		
To construct $\tilde \sigma_h$, consider the pure traction PDE \meqref{e0} with $f=\tilde v_h$ with the
 solution $u\in H^2(\Omega)$. Let
\a{ \sigma = A^{-1} \epsilon(u) \in H^1(\Omega). }
Then \an{\lab{s-v} \d \sigma =\tilde v_h, \quad \|\sigma  \|_{H{(\d)}}
     \le C\|\tilde v_h\|_0. }

     \begin{figure}[htb]
   \begin{center}\setlength{\unitlength}{0.8pt}
   \begin{picture}(380,120)(0,0)
   \multiput(0,0)(130,0){3}{\begin{picture}(40,40)(0,0)
        \multiput(0,0)(40,0){3}{\multiput(0,0)(0,40){3}{\sq}} \end{picture}}
  \def\sqs{\begin{picture}(40,40)(0,0)
   \put(1,1){\line(1,0){38}}\put(1,39){\line(1,0){38}}
   \put(1,1){\line(0,1){38}}\put(39,1){\line(0,1){38}} \end{picture}}
   \multiput(0,0)(130,0){3}{\begin{picture}(40,40)(0,0)
     \multiput(0,0)(40,0){3}{\multiput(0,0)(0,40){3}{\sqs}}\end{picture}}
   \multiput(20,20)(20,0){5}{\multiput(0,0)(0,20){5}{\circle*{4}}}
      \multiput(150,40)(40,0){3}{\multiput(0,0)(0,40){2}{\circle{5}}}
      \multiput(170,20)(40,0){2}{\multiput(0,0)(0,40){3}{\circle{5}}}
      \put(172,23){\tiny $m_E$}
   \multiput(280,20)(40,0){3}{\multiput(0,0)(0,40){3}{\circle{5}}}
   \thicklines
   \put(260,20){\line(1,0){40}}
  \put(280,0){\line(0,1){40}}
  \put(281,22){\tiny $m_c$}
   \end{picture}
   \end{center}
\caption{ \lab{3I}
   Interpolation nodes for the Scott-Zhang $C^0$-$Q_1$ $I_h$,
    $I_{h}^E$, cf. \eqref{E}  and $I_{h}^c$, cf. \eqref{M}.   }
\end{figure}
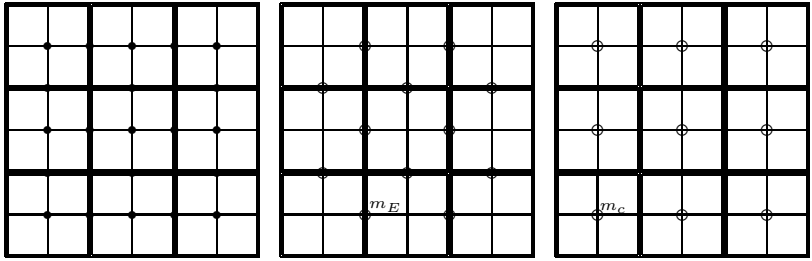
For the analysis, we need a mass-preserving quasi-interpolation operator.  This will be achieved in four steps.
First, let $I_h$  be  the boundary-condition preserving Scott-Zhang operator from \cite{Scott-Zhang}, which
interpolates $H^1_0$ functions to $C^0_0(\Omega)$-$Q_1(\Cal T_h)$  functions, shown in Figure \mref{3I}.
Then,  we correct the mid-point values of edges  of  macro-elements  to get
a mass-preserving on each edge of each macro-element $K_{2i,2j,2h}$.  Let $m_E$ be the
mid-point of edge $E$ of $K_{2i,2j,2h}$, which is also a vertex of $\mathcal{T}_h$, define the associated  nodal basis function
 of the conforming bilinear element by
 $$
 \theta_E(m_E)=1,
 \quad    \theta_E(q)=0 \text{ for other vertexes $q$  of } \mathcal{T}_h.
 $$
 Let
 $$
 c_E=\int_E(v-I_hv)ds\bigg/\int_E\theta_E ds.
 $$
 Define $I_{h}^E:  H_0^1(\Omega)\rightarrow  C^0_0(\Omega)$-$Q_1(\Cal T_h)$ by
 \begin{equation}\label{E}
 I_{h}^Ev=I_{h}v+\sum\limits_{E}c_E\theta_E.
 \end{equation}
 Third, we correct the center value of $I_{h}^Ev$ on each  macro-element.  Let $m_c$
 be the center of macro-element $K_{2i,2j,2h}$, which is also a vertex of $\mathcal{T}_h$.
Let the $Q_1$ nodal basis function $\theta_{ij}$ for vertex $m_c$
  be similarly defined as $\theta_E$.  Define
  $$
  c_{ij}=\big(\int_{E_{1, ij}}(v-I_{2,h}v)ds + \int_{E_{2, ij}}(v-I_{2,h}v)ds\big) \bigg/\big(\int_{E_{1, ij}}\theta_{ij}ds + \int_{E_{2, ij}}\theta_{ij} ds\big),
  $$
  where $E_{1,ij}=[x_{2i}, x_{2i+2}]\times \{y_{2j+1}\}$ and $E_{2, ij}=\{x_{2i+1}\}
    \times [y_{2j}, y_{2j+2}]$ are two intervals in the interior of $K_{2i,2j,2h}$
   that take  $m_c$ as their mid-points,  cf. Figure \mref{3I}.
 Define $I_{h}^c:    H_0^1(\Omega)\rightarrow  C^0_0(\Omega)$-$Q_1(\Cal T_h)$  by
 \begin{equation}\label{M}
 I_{h}^cv=I_{h}^Ev+\sum\limits_{ij}c_{ij}\theta_{ij}.
 \end{equation}
 Finally, define $\tilde{\Pi}_{12}:H^1_0(\Omega)\rightarrow W_{h,0}:=\{w\in W_h| \ \ w(m_e)=0\ \forall e\in \mathcal{E}_h\cap \Gamma\}$ by
 \begin{equation}\label{masspreservein}
 \tilde{\Pi}_{12}v:=\Pi_{12} I_{h}^cv \text{ for any }v\in H_0^1(\Omega).
 \end{equation}
 Since $\int_e \Pi_{12} I_{h}^cv ds =\int_e  I_{h}^cv ds $ for any $e\in\mathcal{E}_h$ ,  the definition of the interpolation operator $\tilde{\Pi}_{12}$ leads to
 \begin{equation}\label{masspreserveproperty}
 \int_E\tilde{\Pi}_{12} vds=\int_Evds\text{ and } \int_{E_{1,ij}}(v-\tilde{\Pi}_{12}v)ds + \int_{E_{2,ij}}(v-\tilde{\Pi}_{12}v)ds=0,
 \end{equation}
 for  any  $E\subset \partial K_{2i,2j,2h}$  and  any macro-element $K_{2i,2j,2h}$. In addition,
 \begin{equation}\label{bounded}
 \|\nabla_h \tilde{\Pi}_{12} v \|_0\leq C\|\nabla v\|_0.
 \end{equation}

Then  $\tilde \sigma_h$ is defined as
\an{ \lab{b-s1}\tilde \sigma_{11} &= \Pi_{11} \sigma_{11}, \\
    \tilde \sigma_{22} &= \Pi_{22} \sigma_{22}, \\
    \lab{b-s} \tilde \sigma_{12} &= \tilde{\Pi}_{12} \sigma_{12},}
 where $\Pi_{11}$, $\Pi_{22}$ and $\tilde{\Pi}_{12}$ are defined in
  \meqref{s11}, \meqref{s22} and \meqref{masspreservein}, respectively.

We verify next, for $\tilde\sigma_h$ defined in \meqref{b-s1}--\meqref{b-s},
\a{   \int_{K_{2i,2j,2h}} (\d_h \tilde\sigma_h-v_h)\cdot
       \phi^r_{m,ij}dx\,dy = 0, \quad m=1,2,3, }
 for $0\le i, j < N/2$.
Note that $\d_h\tilde\sigma_h\ne \tilde v_h$ in general,
   though $\d \sigma =\tilde v_h$. From \meqref{s11}, \meqref{s22}, \meqref{masspreservein} and  \meqref{masspreserveproperty}, and integrations
    by parts it follows
\a{   & \int_{x_{2i}}^{x_{2i+2}}\int_{y_{2j}}^{y_{2j+2}}
        \d_h( \sigma -\tilde \sigma_h)\cdot \phi^r_{1,ij} dy\, dx \\
   &= \int_{y_{2j}}^{y_{2j+2}}(I-\Pi_{11}) [ \sigma_{11}(x_{2i+2},y)
            -\sigma_{11}(x_{2i},y) ] \,dy \\
  &\quad  +
       \int_{x_{2i}}^{x_{2i+2}} (I-\tilde{\Pi}_{12})[ \sigma_{12}(x,y_{2j+2})
            -\sigma_{12}(x,y_{2j}) ] \,dx =0.
   }
   Symmetrically, \a{ \int_{x_{2i}}^{x_{2i+2}}\int_{y_{2j}}^{y_{2j+2}}
       \d_h ( \sigma -\sigma_h )\cdot \phi^r_{2,ij} dy\, dx =0. }
For the last preserved value, as $\d \sigma=\tilde v_h$ pointwise,  from \meqref{s11}, \meqref{s22}, \meqref{masspreservein} and  \meqref{masspreserveproperty},  and integrations
    by parts  it follows
  \a{  &	\int_{x_{2i}}^{x_{2i+2}}\int_{y_{2j}}^{y_{2j+2}}
       \d_h (\sigma -\tilde{\sigma}_h )\cdot \phi^r_{3,ij} dy\, dx\\
        =&  \int_{y_{2j+1}}^{y_{2j+2}}(I-\Pi_{11}) [ \sigma_{11}(x_{2i+2},y)
               -\sigma_{11}(x_{2i},y) ] \,dy  \\
	&\quad -\int_{y_{2j}}^{y_{2j+1}}(I-\Pi_{11}) [ \sigma_{11}(x_{2i+2},y)
               -\sigma_{11}(x_{2i},y) ] \,dy  \\
   &\ + \int_{x_{2i}}^{x_{2i+1}}(I-\Pi_{22}) [ \sigma_{22}(x,y_{2j+2})
               -\sigma_{22}(x,y_{2j}) ] \,dx   \\
   &\quad -    \int_{x_{2i+1}}^{x_{2i+2}}(I-\Pi_{22}) [ \sigma_{22}(x,y_{2j+2})
            -\sigma_{22}(x,y_{2j}) ] \,dx\\
   & \ + \int_{x_{2i}}^{x_{2i+2}}(I-\tilde{\Pi}_{12}) [ \sigma_{12}(x,y_{2j+2})
             +\sigma_{12}(x,y_{2j})  -2\sigma_{12}(x,y_{2j+1}) ] \,dx  \\
   & \ + \int_{y_{2j}}^{y_{2j+2}}(I-\tilde{\Pi}_{12}) [ 2\sigma_{12}(x_{2i+1},y)
                -\sigma_{12}(x_{2i+2},y)-\sigma_{12}(x_{2i},y) ] \,dy=0 \,.
                 }

Thus
  \a{   \Big[v_h - \d_h   \tilde \sigma_h\Big]_{K_{2i,2j,2h}}
           \perp R_{ij} \quad 0\le i,j<N/2. }
We match next $[v_h - \d_h   \tilde \sigma_h ]$
     on each macro-element $K_{2i,2j,2h}$ by
    the divergence of  internal 5 dof of discrete stress:
 \a{  \bar\sigma_{11,2i+1,2j+\frac12}, \
     \bar\sigma_{11,2i+1,2j+\frac32}, \
     \bar\sigma_{12,2i+1,2j+\frac12},\
     \bar\sigma_{22,2i+\frac12,2j+1} \ \hbox{ and }
     \bar\sigma_{22,2i+\frac32,2j+1}, }
 where $\bar{\sigma}_{11,2i+1,2j+\frac12}$ denotes the value of $\bar\sigma_{11}$ at $((2i+1)h, (2j+\frac12)h)$ and other notations
  are defined similarly. Note that the four mid-edge values of $\bar \sigma_{12}$ are the same.
Here on each macro-element, $[v_h - \d_h   \tilde \sigma_h ]$ is
   in the following space
\an{ \lab{s-M} M_{ij} = \operatorname{span}
    \{ \phi^c_{m,ij}, \ m=1,2,3,4,5\} }
where $\phi^c_{m,ij}$ are defined in Figure \mref{5-m}.

\begin{figure}[htb]
   \begin{center}\setlength{\unitlength}{1.4pt}
   \begin{picture}(220,43)(0,0)
\def\sq{\begin{picture}(40,40)(0,0)
   \multiput(0,0)(0,20){3}{\line(1,0){40}}
   \multiput(0,0)(20,0){3}{\line(0,1){40}}\end{picture}}

  \multiput(0,0)(45,0){5}{\sq}
\put(0,0){\begin{picture}(60,60)(0,0)
   \put(12,29){\makebox(-6,0){$\p{0\\0}$}}   \put(32,29){\makebox(-6,0){$\p{0\\0}$}}
      \put(12,9){\makebox(-6,0){$\p{-1\\0}$}} \put(32,9){\makebox(-6,0){$\p{1\\0}$}}
     \end{picture}}
\put(45,0){\begin{picture}(60,60)(0,0)
   \put(12,29){\makebox(-6,0){$\p{1\\0}$}}   \put(32,29){\makebox(-6,0){$\p{-1\\0}$}}
      \put(12,9){\makebox(-6,0){$\p{0\\0}$}} \put(32,9){\makebox(-6,0){$\p{0\\0}$}}
     \end{picture}}
\put(90,0){\begin{picture}(60,60)(0,0)
    \put(12,29){\makebox(-6,0){$\p{0\\0}$}}   \put(32,29){\makebox(-6,0){$\p{1\\0}$}}
      \put(12,9){\makebox(-6,0){$\p{-1\\-1}$}} \put(32,9){\makebox(-6,0){$\p{0\\1}$}}
     \end{picture}}
\put(135,0){\begin{picture}(60,60)(0,0)
    \put(12,29){\makebox(-6,0){$\p{0\\0}$}}   \put(32,29){\makebox(-6,0){$\p{1\\1}$}}
      \put(12,9){\makebox(-6,0){$\p{-1\\-1}$}} \put(32,9){\makebox(-6,0){$\p{0\\0}$}}
     \end{picture}}
\put(180,0){\begin{picture}(60,60)(0,0)
    \put(12,29){\makebox(-6,0){$\p{0\\1}$}}   \put(32,29){\makebox(-6,0){$\p{0\\0}$}}
      \put(12,9){\makebox(-6,0){$\p{0\\-1}$}} \put(32,9){\makebox(-6,0){$\p{0\\0}$}}
     \end{picture}}
   \end{picture}
   \end{center}
\caption{ \lab{5-m}    Nodal values of  basis functions
   $\{\phi^c_{m,ij}, \ 1\le m\le 5\}$ in $M_{ij}$
    on macro-element $x_{2i}\le x\le x_{2i+2}$,
    $y_{2j}\le y\le y_{2j+2}$, cf. \meqref{s-M}. }
\end{figure}
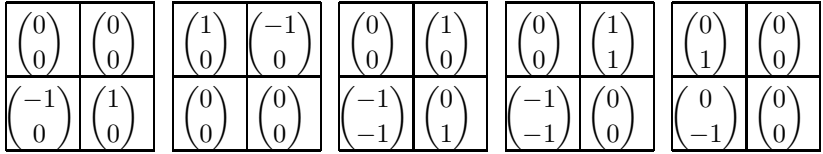

\mcomment{
restart: n:=2: eq1:=0:eq2:=0:eq3:=0:
for i from 1 to n do for j from 1 to n do eq1:=eq1+u[i,j]:
   eq2:=eq2+v[i,j]: eq3:=eq3+u[i,j]*(j-1/2)/n-v[i,j]*(i-1/2)/n: od: od:
ss:=solve({eq1=0,eq2=0,eq3=0},{u[n,n],u[n,n-1],v[n,n]}):

for i from 1 to n do a[1,i]:=0: a[n+1,i]:=0: b[i,1]:=0: b[i,n+1]:=0: od:
for i from 1 to n+1 do c[i,1]:=0: c[1,i]:=0: c[i,n+1]:=0: c[n+1,i]:=0: od:
es:=0:
for i from 1 to n do for j from 1 to n do
   e1[i,j]:=(a[i+1,j]-a[i,j])+((c[i+1,j+1]+c[i,j+1])/2-(c[i+1,j]+c[i,j])/2)=u[i,j]:
   e2[i,j]:=(b[i,j+1]-b[i,j])+((c[i+1,j+1]+c[i+1,j])/2-(c[i,j+1]+c[i,j])/2)=v[i,j]:
   es:=es,e1[i,j],e2[i,j]: od: od:
uns:=c[2,2]:
for i from 2 to n do for j from 1 to n do uns:=uns,a[i,j],b[j,i]: od: od:
for i from 2 to n do for j from 2 to n do uns:=uns,c[i,j]: od: od:
e1:=subs(ss,{es})[]:
ss:=u[1,1]=0: for i from 1 to n-1 do for j from 1 to n do ss:=ss,u[i,j]=0,v[j,i]=0: od: od:
e2:=subs(ss,{e1})[]:
sc:=0: for i from 1 to n-2 do for j from 1 to n-2 do sc:=sc,c[i+1, j+1]=0: od: od:
solve({e2,sc},{uns});

eq1:=u11+u12+u13+u21+u22+u23+u31+u32+u33=0,v11+v12+v13+v21+v22+v23+v31+v32+v33=0,
     (1/6)*(u11+u21+u31)+(3/6)*(u12+u22+u32)+(5/6)*(u13+u23+u33)
    =(1/6)*(v11+v12+v13)+(3/6)*(v21+v22+v23)+(5/6)*(v31+v32+v33);
eq2:=(a21*3+c22*3/2)=u11,((a31-a21)*3+(c22+c32)*3/4)=u21,
     (-a31*3+c32*3/2)=u31,
     (a22*3+(c23-c22)*3/2)=u11,((a31-a21)*3+(c22+c32)*3/4)=u21,
     (-a31*3+c32*3/2)=u31,
   }

On each macro-element,  define 5 stress functions to match the
   5 basis functions of $M_{ij}$ such that
  \a{  \d_h \sigma_{m,ij} =  \phi^c_{m,ij}. }
Each such a function is denoted by a vector of its nodal values:
\a{ \frac 1h \sigma_{m,ij}= \p{ \bar{\sigma}_{11,2i+1,2j+1/2} \\
                                 \bar{\sigma}_{11,2i+1,2j+3/2} \\
				 \bar{\sigma}_{12,2i+1/2,2j+1/2} \\
				\bar{\sigma}_{22,2i+1/2,2j+1}\\
                                 \bar{\sigma}_{22,2i+3/2,2j+1} }
	=\p{ -1\\0\\0\\0\\0},
	  \p{0\\1\\0\\0\\0}, \frac12 \p{-1\\-1\\-1\\-1\\1},
	  \frac 12 \p{-1\\-1\\-1\\-1\\-1},\frac 12 \p{0\\0\\0\\-1\\0},
	  }
for $1\le m\le 5$. A linear expansion $  \Big[v_h - \d_h   \tilde \sigma_h\Big]_{K_{2i,2j,2h}}
    =\sum\limits_{m=1}^5 c_{m,ij} \phi^c_{m,ij}$ defines $\bar \sigma_h (x,y)$ by
\a{  \bar \sigma_h (x,y)=\sum_{m=1}^5 c_{m,ij}\sigma_{m,ij},
	\quad (x,y)\in K_{2i,2j,2h}. }
Thus
\begin{equation}\label{s20}
 \d_h \bar \sigma_h = v_h -  \d_h   \tilde{\sigma}_h \text{ and } \| \bar \sigma_h \|_0 \le C\|v_h\|_0.
\end{equation}
 This stability is obtained by the standard scaling argument as
   all norms on 5-dimensional space $M_{ij}$ are equivalent.

The final $\sigma_h$ for $v_h$ is defined as
 \a{ \sigma_h = \bar \sigma_h+   \tilde \sigma_h. }
As $\d_h \sigma_h=v_h$,  by \meqref{bounded} and \meqref{s20},
   the discrete B-B condition holds uniformly.
\end{proof}

\section{Numerical tests}\lab{s-numerical}

Two examples in 2D and one in 3D are presented to demonstrate the methods.
These  are pure displacement problem with a homogeneous boundary condition
         that $u\equiv 0$ on $\partial\Omega$. Assume the material is isotropic in the sense that
   \an{ \lab{A}
      A \sigma &= \frac 1{2\mu} \left(
	   \sigma - \frac{\lambda}{2\mu + n \lambda} \operatorname{tr}(\sigma)
		\delta \right), \quad n=2,3, }
  where $\delta=\p{1 &0\\0&1}$, and $\mu$ and $\lambda$ are the
	Lam\'e constants such that $0<\mu_1\le \mu \le \mu_2$ and
    $0<\lambda<\infty$.

In 2D, let  the  exact solution on the unit square $[0,1]^2$ be
   \e{\lab{e-1}   u= \p{4 x(1-x)y(1-y)\\ -4 x(1-x)y(1-y)}, }
 and
   \e{\lab{e-2}   u= \p{ e^{x-y}x (1-x) y (1-y)\\
	\sin(\pi x)\sin(\pi y)}. }
Notice  that the second example is from \cite{Yi}.

In 2D, the parameters in \meqref{A} are chosen as  \a{ \lambda=1 \quad \hbox{ and } \quad
	  \mu = \frac 12. }
Then, the true stress function $\sigma$
     and the load function $f$ are defined by the equations in
	\meqref{eqn1},   for the given  solution $u$.

 \begin{table}[htb]
  \caption{ The error and the order of convergence, for \meqref{e-1}.}
\lab{b-1}
\begin{center}  \begin{tabular}{c|rr|rr|rr}  
\hline & $ \|I_h u- u_h\|_{0}$ &$h^n$ &
    $ \|I_h\sigma -\sigma_h\|_{0}$ & $h^n$  &
    $ \|\d(I_h\sigma -\sigma_h)\|_{0}$ & $h^n$  \\ \hline
 1&  0.05893&0.0&  0.72887&0.0&   1.41421356&0.0\\
 2&  0.02447&1.3&  0.24585&1.6&   0.35355339&2.0\\
 3&  0.00714&1.8&  0.06587&1.9&   0.08838835&2.0\\
 4&  0.00190&1.9&  0.01708&1.9&   0.02209709&2.0\\
 5&  0.00048&2.0&  0.00440&2.0&   0.00552427&2.0\\
 6&  0.00012&2.0&  0.00113&2.0&   0.00138106&2.0\\
 7&  0.00003&2.0&  0.00029&2.0&   0.00034526&2.0 \\
      \hline
\end{tabular}\end{center} \end{table}

In the computation, the level one grid is the given domain, a unit
   square or a unit cube.
Each grid is refined into a half-size grid uniformly, to get
   a higher level grid, see the first column in Table \mref{b-1}.
In Table \mref{b-1}, the errors and the convergence order
   in various norms are listed  for the true solution \meqref{e-1}.
Here and in rest tables in the section,  $I_h$ is the usual nodal
   interpolation operator. For example,
  $I_hu_1(x_{i}+h/2, y_{j}+h/2)
	= u_1(x_{i}+h/2, y_{j}+h/2)$,
   $I_h \sigma_{11}
        (x_{i}, y_{j}+h/2)
	=\sigma_{11}
        (x_{i}, y_{j}+h/2)$, and $I_h\sigma_{12}=\Pi_{12}\sigma_{12}$,
    defined in \meqref{Pi12}.
An order 2 convergence is observed  for both displacement and stress,
  see Table \mref{b-1}.
However,  Theorem \ref{MainError} only shows the first order convergence.
Further studies on this superconvergence should be performed.

 The next example, \meqref{e-2}, of Yi \cite{Yi} is implemented for a
comparison.
The finite element errors and the order of convergence are listed
  in Table \mref{b-2}.
An order 2 convergence is again observed.
Notice  that, see Figure \mref{comparison}, the minimal  element of this paper
  has a much less dof than that of Yi,  but has one order higher
    of convergence.

 \begin{table}[htb]
  \caption{ The error and the order of convergence, for \meqref{e-2}.}
\lab{b-2}
\begin{center}  \begin{tabular}{c|rr|rr|rr}  
\hline & $ \|I_h u- u_h\|_{0}$ &$h^n$ &
    $ \|I_h\sigma -\sigma_h\|_{0}$ & $h^n$  &
    $ \|\d(I_h\sigma -\sigma_h)\|_{0}$ & $h^n$  \\ \hline
 1&  0.03619&0.0&  3.08021&0.0&  12.20143741&0.0\\
 2&  0.09843&0.0&  0.54275&2.5&   2.36338456&2.4\\
 3&  0.02594&1.9&  0.15169&1.8&   0.63139891&1.9\\
 4&  0.00664&2.0&  0.03964&1.9&   0.16050210&2.0\\
 5&  0.00167&2.0&  0.01014&2.0&   0.04029305&2.0\\
 6&  0.00042&2.0&  0.00258&2.0&   0.01008376&2.0\\
      \hline
\end{tabular}\end{center} \end{table}

As a third example,  we compute a 3D solution for
   the following exact solution:
  \e{\lab{e-3} u=\p{ 16x(1-x)y(1-y)z(1-z) \\
                     32x(1-x)y(1-y)z(1-z) \\
                     64x(1-x)y(1-y)z(1-z)}, }
on the unit cube $[0,1]^3$.
This time,  the parameters  in \meqref{A} are  taken as
  \a{ \lambda = 1, \quad \mu=\frac 12 \ \hbox{ and } \  n=3. }
Again the order of convergence is still one higher than what is proved
   in this paper,  see Table \mref{b-3}.

 \begin{table}[htb]
  \caption{ The error and convergence in 3D, for \meqref{e-3}. }
\lab{b-3}
\begin{center}  \begin{tabular}{c|rr|rr|rr}  
\hline & $ \|I_h u- u_h\|_{0}$ &$h^n$ &
    $ \|I_h\sigma -\sigma_h\|_{0}$ & $h^n$  &
    $ \|\d(I_h\sigma -\sigma_h)\|_{0}$ & $h^n$  \\ \hline
 1&  0.16366&0.0&  3.64496&0.0&  8.94883415&0.0 \\
 2&  0.07716&1.1&  0.89446&2.0&  1.73418255&2.4 \\
 3&  0.02332&1.7&  0.23153&1.9&  0.42577123&2.0 \\
 4&  0.00628&1.9&  0.05946&2.0&  0.10668050&2.0 \\
 5&  0.00161&2.0&  0.01518&2.0&  0.02628774&2.0 \\
      \hline
\end{tabular}\end{center} \end{table}

As the last example,  we compute the pure traction problem \meqref{e0}
 with the exact solution
\an{\lab{s-l} u= \left[100 x^2 (1-x)^2 y^2 (1-y)^2 -\frac 19 \right]
    \p{1\\-1}. }
The matrix $A$ is same as that in the first two examples.
Our new finite element has no problem in solving the pure
  traction problems.
The convergence results are listed in Table \ref{t-s-l}.

 \begin{table}[htb]
  \caption{\lab{t-s-l}
 The errors and the order of convergence for
   the pure traction problem \meqref{s-l}. }
\begin{center}  \begin{tabular}{c|rr|rr|rr}  
\hline & $ \|I_h u- u_h\|_{0}$ &$h^n$ &
    $ \|I_h\sigma -\sigma_h\|_{0}$ & $h^n$  &
    $ \|\d(I_h\sigma -\sigma_h)\|_{0}$ & $h^n$  \\ \hline
 2&  0.41470&0.0&  1.19604&0.0&   4.14320380&0.0 \\
 3&  0.12546&1.7&  0.26426&2.2&   1.10584856&1.9 \\
 4&  0.03273&1.9&  0.06572&2.0&   0.28799493&1.9 \\
 5&  0.00827&2.0&  0.01648&2.0&   0.07297595&2.0 \\
 6&  0.00207&2.0&  0.00412&2.0&   0.01830958&2.0 \\
 7&  0.00052&2.0&  0.00103&2.0&   0.00458156&2.0 \\
      \hline
\end{tabular}\end{center} \end{table}

\end{document}